\documentclass[11pt]{article}

\usepackage[hypertex]{hyperref}
\usepackage[active]{srcltx}

\usepackage{epsfig}

\usepackage[active]{srcltx}

\usepackage{amssymb}
\usepackage{amscd}
\usepackage{amsmath}
\usepackage{amsfonts}
\usepackage{amsthm}
\usepackage{mathrsfs}

\newtheorem{theo}{Theorem}[section]
\newtheorem{prop}[theo]{Proposition}
\newtheorem{lem}[theo]{Lemme}
\newtheorem{cor}[theo]{Corollary}

\theoremstyle{remark}

\newcommand{\R}{{\mathbb{R}}}

\newcommand{\C}{{\mathbb{C}}}
\newcommand{\Z}{{\mathbb{Z}}}

\newcommand{\la}{\lambda} 
\newcommand{\al}{\alpha} 
\newcommand{\be}{\beta} 
\newcommand{\Om}{\Omega} 
\newcommand{\te}{\theta} 
\newcommand{\ga}{\gamma}
\newcommand{\om}{\omega} 
\newcommand{\Si}{\Sigma} 
\newcommand{\si}{\sigma} 

\newcommand{\Ci}{{\mathcal{C}}^{\infty}}

\newcommand{\Lie}{{\mathfrak{g}}} % alg{\`e}bre de Lie
\newcommand{\Ad}{\operatorname{Ad}} % action adjointe
\newcommand{\id}{\operatorname{id}} % identite 
\newcommand{\Hom}{\operatorname{Hom}}

\newcommand{\trace}{\operatorname{tr}}

\newcommand{\alt}{\operatorname{alt}}

\newcommand{\Alc}{{\mathfrak{A}}}   %%% Alcove de Weyl

\newcommand{\Hilb}{{\mathcal{H}}}   %% Espace de Hilbert

\newcommand{\mtore}{{\operatorname{T}}}  %%% tore maximal
\newcommand{\Liet}{{\mathfrak{t}}}   %%% son alg{\`e}bre de Lie
\newcommand{\Lat}{\Lambda} %%% son r{\'e}seau entier 

\newcommand{\U}{\operatorname{U}}  %%% groupe unitaire
\newcommand{\Hp}{{\mathcal{H}}_+}  %%% demi-plan de Poincar{\'e}.

\newcommand{\Sl}{\operatorname{Sl}}

\newcommand{\omh}{\Om^{1/2}}

\newcommand{\proj}{{\mathbb{P}}}

\newcommand{\Sp}{\operatorname{Sp}}
\newcommand{\Mp}{\operatorname{Mp}}

\newcommand{\even}{\operatorname{ev}}
\newcommand{\odd}{\operatorname{odd}}

\newcommand{\ind}{\operatorname{ind}}

\newcommand{\Mo}{\Gamma} 

\title{Asymptotic properties of the quantum representations of the
  modular group} 

\author{Laurent CHARLES  \footnote{Institut de
    Math{\'e}matiques de Jussieu (UMR 7586), Universit{\'e} Pierre et
    Marie Curie -- Paris 6, Paris, F-75005 France.}}

\begin{document}

\maketitle

\begin{abstract}
We study the asymptotic behaviour of the quantum representations of
the modular group in the large level limit. We prove that each
element of the modular group acts as a Fourier integral
operator. This provides a link between the classical and quantum
Chern-Simons theories for the torus. From this result we deduce the known asymptotic expansion of the Witten-Reshetikhin-Turaev
invariants of the torus bundles with hyperbolic monodromy.
\end{abstract}

\maketitle

\bibliographystyle{plain}

Quantum Chern-Simons theory was introduced twenty years ago by
Witten \cite{Wi} and Reshetikhin-Turaev \cite{ReTu}. It provides among other things invariants of
three-dimensional manifold and representation of the mapping class
group of surfaces, cf. \cite{BaKi}, \cite{Tu}  for an exposition of the theory and
\cite{Fr} for a survey on recent developments. 
This theory has a semi-classical limit, where the level, an integral 
parameter denoted by $k$, plays the role of the inverse of the Planck
constant. 
In this paper, we are concerned with the torus and its mapping class
group, $\Sl ( 2, \Z )$. We study the large $k$  behaviour of the quantum
representation of the modular group.

The quantum representations may be equivalently defined with algebraic or geometrical
methods. Geometrically, we consider a line bundle, called the
Chern-Simon bundle, over the moduli space of flat $G$-principal bundles on the
torus. Here $G$ is a compact Lie group that we assume to be simple and 
simply connected. Then the modular group acts linearly on the
space of holomorphic sections of the $k$-th tensor power of the
Chern-Simons bundle. 

A part of this construction is standard in geometric quantization: to
any compact K{\"a}hler manifold with an integral fundamental form one
associates the space of holomorphic
sections of a prequantum bundle. In this general context the usual tools of
microlocal analysis have been introduced. In particular given a prequantum
bundle automorphism, we define a class of operators
similar to the Fourier integral operators \cite{oim_qm},
\cite{oim_LS} quantizing it. Our main result, theorem \ref{theo:FIO_Modul}, says that
each element of $\Sl ( 2, \Z)$ acts as a Fourier integral operator on
the quantum spaces, the underlying action on the Chern-Simons bundle
being defined through gauge theory. This establishes a clear link between the
quantum and classical Chern-Simons theories.

As a corollary, we can estimate the character of the quantum
representations of the hyperbolic elements of $\Sl ( 2, \Z)$. More
generally, under a transversality assumption, one proves that the trace
of a Fourier integral operators has an asymptotic expansion, which
generalizes in some sense the Lefschetz fixed point formula
\cite{oim_LS}. 
The characters of the quantum representation of the modular group are the three-dimensional invariants of the torus
bundles. In this way we recover the asymptotic expansion proved by Jeffrey
\cite{Je}, whose leading term is given in terms of the Chern-Simons
invariants and the torsion of some flat bundles over the torus bundle. The proof in \cite{Je}
is completely different and 
relies on the reciprocity formula for Gauss sum. Our result is slightly
more general since we treat any hyperbolic element with any simple
simply connected group $G$. But, what is more important, we hope that
our analytic method will work in other cases. In the companion paper \cite{oim_MCG}, we
prove similar result for the mapping class group in genus $\geqslant
2$. 
 
Besides the semiclassical results, we also give a careful construction
of the quantum representations, comparing the geometric and algebraic
methods. Strictly speaking, we do not have representations of the modular group but only projective representations which lift to 
genuine representations of the appropriate extension of $\Sl (2,
\Z)$. The extensions appearing naturally are not the same in the
geometric and the algebraic approach.

The paper is organised as follows. In section \ref{sec:asympt-expans-torus}, we state our result about the asymptotic
expansion of the trace of the quantum representations.  In section
\ref{sec:symplectic-datas}, we introduce the phase space of the
Chern-Simons theory for the torus, its symplectic structure and
prequantum bundle. The relation with gauge theory is the content of
section \ref{sec:chern-simons-theory}. In section \ref{sec:quantization}, we introduce a
complex structures on the phase space and the associated 
quantum Hilbert spaces. We exhibit basis in terms of theta
functions. The modular group acts naturally on the previous datas but
does not preserve the complex structure. In section
\ref{sec:geom-quant-repr} we identify the quantum spaces associated to the various complex
structures. This leads to the definition of the quantum
representations. In section \ref{sec:modul-tens-categ} we compare these representations with the
ones defined by algebraic methods. The next two sections are devoted to
semi-classical results: section \ref{sec:compl-struct-depend} on the identification of the quantum
spaces and section \ref{sec:asympt-prop-quant} on the quantum
representations. In a first appendix we prove basic facts on
theta functions. In a second appendix we list some notations used in
the paper.

\section{Characters of the quantum representations} \label{sec:asympt-expans-torus}

Let $G$ be a compact simple and simply connected Lie group. The phase
space of the Chern-Simons theory for an oriented surface $\Si$ and
group $G$ is the moduli
space of flat $G$-principal bundles over $\Si$. For a torus, this moduli space
identifies with the quotient $\mtore ^2 /W$, where $\mtore$ is a maximal torus of
$G$ and $W$ is the Weyl group acting diagonally. The modular group
$\Sl ( 2, \Z)$, being the mapping class
group of the torus, acts on this moduli space. 
More explicitly, since $\mtore = \Liet /
\Lat$, with $\Liet$ the Lie algebra of $\mtore$ and $\Lat$ the
integral lattice, we have a bijection $\mtore^2 / W \simeq \Liet^2 / ( \Lat^2 \rtimes
W)$. Identify $\Liet^2$ with $\R^2 \otimes \Liet $, then an element $A
\in \Sl ( 2, \Z)$ acts on the moduli space by sending the class of
$x$ to the class of $(A \otimes \id_{\Liet}  ). x$.

Applying geometric
quantization, we obtain a family of projective representations of the modular group indexed by a positive integer $k$. Since the construction is rather long, we only give in this introduction the representation of the generators  $$ S =  \left(
  \begin{array}{cc} 0 & -1   \\  1  &  0   \end{array} \right) ,
\qquad T =  \left(
  \begin{array}{cc} 1 & 1  \\  0 & 1 \end{array} \right)  $$
of the modular group.  
Let $B$ be the basic inner product of the Lie algebra of $G$. We choose a set of
positive roots and denote by  $\Alc \subset \Liet $ the corresponding
open fundamental Weyl alcove.  We identify the weight lattice $\Lat^*
\subset \Liet^*$ with a lattice  of $\Liet $ via the basic inner
product. The $k$-th representation has a particular basis indexed
by the set $\Alc \cap k^{-1} \Lat^*$. 
For any $\la, \mu \in \Alc \cap k^{-1} \Lat^* $,
let  
$$ t_{\la \mu} =  \delta_{\la,\mu}  \exp ( i \pi k B( \la, \la ) )$$
and 
$$ s_{\la \mu } =   i^p k ^{-\frac{n}{2}}
  \operatorname{Vol} ^{-1}  ( \Liet / \Lat )   \sum_{w \in W } (-1)^{\ell
  (w) }\exp(-  2 i \pi k B( \la , w(\mu) ) )  $$
 where $n$ is the rank of $G$ and $p$ is the integral part of
 $n/2$. 

If the rank of $G$ is even, the map sending $S$ and $T$ to
the matrices $s_{\la \mu}$  and $t_{\la \mu}$ extends to a unitary
representation $R^k_{\even}$ of the modular group. If the rank of $G$ is odd, we
obtain a representation $R^k _{\odd}$ of an extension $\Mp ( 2, \Z)$
of the modular group by
$\Z / 2 \Z$. 
Recall that the metaplectic group $\Mp ( 2, \R)$ is the connected
two-cover of $\Sl ( 2, \R)$. Then $\Mp(2, \Z)$ is defined as the subgroup of the
metaplectic group consisting of the elements which project onto the
modular group. The representation $R^k_{\odd}$ of the elements
projecting onto $S$ and
$T$ is given by the matrices $\pm s_{\la \mu}$ and $\pm t_{\la
  \mu}$. 

\begin{theo} Assume the rank of $G$ is even, then for any hyperbolic
  element $A \in \Sl (2, \Z)$, we have 
$$ \trace  R^k_{\even} (A) = \frac{ (-1)^{2 \epsilon p} }{|W|}
\sum_{\substack{ w \in W, \; x \in \mtore ^2 / \\ (A \otimes w).x = x }}
(-1)^{\ell (w)} 
\frac{ e ^{ ik \theta ( A \otimes w , x ) }}{ | \det ( \id - A \otimes
  w ) | ^{1/2}} + O( k^{-1}) $$
where 
\begin{itemize} 
\item $ \epsilon =0$ if the trace of $A$ is bigger than
$2$ and $\epsilon =1$ otherwise. 
\item $\theta (A \otimes w , x ) = \pi ( B( \mu, p ) - B ( \ga , q) +
  B( \ga , \mu))$ if $x \in \mtore^2 \simeq ( \Liet/\Lat)^2$ is the class of  $(p,q) \in \Liet ^2$ and $(\ga, \mu) = (A \otimes w)(p,q) - (p,q)$. 
\end{itemize}
If the rank of $G$ is odd, let $\tilde {A} \in \Mp ( 2, \R)$
projecting onto an hyperbolic element $A$ of the modular group. Then
the same result holds for the trace of $
R^k_{\odd} (\tilde{A})$ except that the equivalent has to be
multiplied by $\exp ( i \frac{\pi}{2}\ind (\tilde A))$, where
$\ind ( \tilde A) \in \Z$ modulo $4 \Z$.    
\end{theo}

It is a general property of topological quantum field theories that the
trace of the quantum representation of an element $A$ of the modular group is the
invariant of the  mapping torus 
$$ M_A := \bigl((\R^2 / \Z^2 ) \times \R \bigr)/ (A y,t) \sim
(y, t+1). $$  
Here we ignore the the complications due to the framing of 3-dimensional manifold and the related fact that we
only have a projective representation.  
For any $(x_1,x_2) \in \mtore^2$ and $w \in W$ such that $ (A \otimes w).(x_1,x_2) =
(x_1,x_2)$, consider the flat $G$-principal bundle  $P_A \rightarrow M_A$
whose holonomies along the paths $\ga (s) = [s,0,0]$ , $[0,s,0]$ and
$[0,0,s]$ are respectively $x_1$, $x_2$ and $w^{-1}$. Then 
$ (2 \pi )^{-1} \theta (A \otimes
w, x)$ is the Chern-Simons invariant of $P_A$. This is in agreement with the formula obtained by Witten
using the Feynman path integral (heuristic) definition of the
three-dimensional invariants. We refer the reader to Jeffrey's paper
\cite{Je} for more details. In particular the factor $|\det ( \id
- A \otimes w)|^{-1}$ appears as an integral of the torsion of the
adjoint bundles over the moduli space of flat $G$-principal bundle over $M_A$. 

\section{Lie group notations} \label{sec:notation}

Let $\Lie$ be a compact simple Lie algebra, and $G$ the
corresponding compact connected and simply-connected Lie group. Choose
a maximal torus $\mtore$ of $G$ and denote by $\Liet$ its Lie
algebra. The integral lattice $\Lat$ of $\Liet$ is defined as the
kernel of the exponential map $ \Liet \rightarrow \mtore$. Since $G$
is simply-connected, $\Lat$ is the lattice of $\Liet$ generated by the
coroots $\al^\vee$ for the (real)
%\footnote{The root  space consists of the vectors $X$ such that
%$[H,X] = 2 i \pi \al (H) X$, for any $H \in \Liet$.}
 roots $\alpha$.

Let the basic inner product $B$ be the unique invariant inner product
on $\Lie$ such that for each long root $\al$, $B( \al^\vee  ,\al^\vee
) = 2 $. Through the paper, we will use $B$ to
identify $\Liet$ with $\Liet^*$. The basic inner product has the important property that it restricts to an integer-valued
$\Z$-bilinear form on $\Lat$ which takes even values on the diagonal. 

We fix a
set $\Delta_+$ of positive roots and let $\Liet_+ $ be the corresponding
positive open Weyl chamber. Let $\al_0$ be the highest root and $\Alc$
be the open fundamental Weyl  alcove
$$ \Alc := \{ \la \in \Liet_+ / \; \al_0 ( \la ) < 1   \}$$
We denote by $W$ the Weyl group of $(G, \mtore)$. Let $\ell : W \rightarrow \{ \pm 1 \}$ be the alternating
character of $W$.

\section{The symplectic data}\label{sec:symplectic-datas}

In this section we endow $\mtore^2$
with a symplectic form $\om$ and a prequantum bundle $L$, that is a complex Hermitian line bundle together
with a connection of curvature $\frac{1}{i} \om$. Furthermore we
introduce commuting actions of the Weyl group and the modular group on $L$.

\subsection{A prequantum bundle on $\Liet^2$}  \label{sec:preq-bundle-Liet2}

Denote by $p$ and $q$ the projections $\Liet^2 \rightarrow \Liet$ on the first and second factor respectively. Let $\om$ be the symplectic form on $\Liet ^2$ given by 
$$ \om = 2 \pi B( dp, dq ).$$
 Consider the trivial complex line bundle $L_{\Liet^2}$ over $\Liet^2$
 with fiber $\C$ and  connection
 $$d + \frac{\pi}{i} ( B( p , dq ) - B( q, dp )).$$ Its curvature
 is $\frac{1}{i} \om$, so it is a prequantum bundle.

\subsection{Heisenberg group and reduction to $\mtore^2$} \label{sec:heis-group-reduct}
Introduce the (reduced) Heisenberg group $\Liet^2 \times \U (1)$ with multiplication
$$ (x, u).(y, v) = ( x + y , u v \exp \bigl( \tfrac{i}{2} \om (x, y)
\bigr) \bigr)$$
The same formula defines an action of the Heisenberg group on
$L_{\Liet^2} = \Liet^2 \times \C$. This action preserves the trivial metric and the
connection. The lattice $\Lat ^2$ embeds into the Heisenberg group
$$ \Lat^2 \rightarrow \Liet^2 \times \U (1), \qquad ( p, q) \rightarrow (
p,q, \exp  ( i \pi B(p,q)) ) $$
Using that $B$ takes integral values on $\Lat$, we prove that this map is a
group morphism. Hence we get an action of $\Lat^2$ on $L_{\Liet^2}$ by automorphisms of prequantum bundle. 

By quotienting, we obtain a symplectic form on $\mtore ^2 = \Liet^2 /
\Lat ^2$ with a prequantum bundle $L := \Liet^2 \times \C / \Lat ^2 $ over
$\mtore ^2$. 

\subsection{Weyl group} 
Consider the diagonal action of the Weyl group $W$ on $\Liet^2$ and lift
this action trivially on the bundle $L_{\Liet^2}$. Since $W$ acts on
$\Liet$ by isometries, $W$ acts on $\Liet^2$ by linear
symplectomorphisms and on $L_{\Liet^2}$ by isomorphisms of
prequantum bundle.

The Weyl group preserves the integral lattice $\Lat$. Consider the
semi-direct product  $ W \rtimes
\Lat^2$ where $W$ acts diagonally on $\Lat^2$. 
Is is easily checked that
the actions of $\Lat^2$ and $W$ on the prequantum bundle over
$\Liet^2$ generate an action of $ W \rtimes
\Lat^2$. Then, quotienting by $\Lat^2$, we
obtain an action of $W = (W \rtimes
\Lat^2)/ \Lat^2$  on $L$. 
Since the action of the Weyl group on the base $\mtore^2$ is not free, we will not
consider the orbifold quotient $\mtore^2 / W$ and its prequantum
bundle. 

\subsection{Modular group} \label{sec:modular-group}
Let us consider the symplectic action of the modular group $\Mo = \Sl(
2, \Z)$ on $\Liet^2 $ given by 
$$ A. (p,q) = ( ap + bq, cp + d q), \qquad  A =  \left(
  \begin{array}{cc} a & b   \\  c  &  d  \end{array} \right) $$
The trivial lift to the prequantum bundle $L_{\Liet^2}$  preserves the metric and the
connection. Furthermore this action together with
the action of $\Lat^2$ define an action of the semi-direct product
$\Mo \rtimes \Lat^2$. To prove this, one has to use that $B
( p,q) $ is integral when $p, q \in \Lat$ and even if furthermore $p
= q $. Consequently we get an action of the modular group on $L$ by
prequantum bundle isomorphisms. Observe that the Weyl group action on
$L$  commute with the modular action.

\section{Chern-Simons theory} \label{sec:chern-simons-theory}

We explain how the definitions of the previous section 
can be deduced from gauge theory. Our aim
is only to motivate the 
constructions. No proof in the paper relies on the gauge theoretic
considerations. 

The phase space of the Chern-Simons theory for an oriented surface
$\Si$ is the moduli space of representations of the fundamental group
of $\Si$ in $G$. When $\Si $ is a torus, the fundamental group is the
free Abelian group with two generators, so each representation
is given by a pair of commuting elements of $G$ unique up to
conjugation.  In the same way that $G/ \operatorname{Ad} G \simeq
\mtore / W$, one shows that these representations are conjugate to a representation in the maximal torus $\mtore$, uniquely
up to the action of the Weyl group. So the moduli space of
representation for the torus is $\mtore^2 / W$. 

\subsection{Gauge theory presentation} \label{sec:gauge-theory-pres}
Consider the space $\Om^{1} ( \Si, \Lie )$ of connections of the
trivial $G$-principal bundle with base $\Si$. It is a symplectic
vector space with symplectic product given by 
$$ \Om ( a, b) = 2 \pi \int _{\Si} B( a,b ) . $$
The gauge
group $\Ci ( \Si , G)$ acts on $\Om^{1} ( \Si, \Lie )$ by symplectic affine isomorphisms:
$$ g.a = \Ad_g a - g^* \bar{\te}$$
where $\bar{\te} \in \Om^1 (G, \Lie)$ is the right-invariant Maurer-Cartan
form. Each gauge class of flat connections is determined by its
holonomy representation. The
quotient of the space of flat connections by the gauge group
may be viewed as a symplectic quotient which defines a symplectic
structure on  the moduli space of representations. 

In the case $\Si$ is a torus, we can avoid this infinite dimensional
quotient proceeding as follows. Represent $\Si$ as the quotient $\R^2 / \Z^2$ with
coordinates $x, y$. Then the map 
$$ \Liet^2 \rightarrow \Om^{1} ( \Si, \Lie) ,\qquad ( p, q )
\rightarrow pdx + q dy $$
is a symplectic embedding where the symplectic product of $\Liet^2$ is
the one of section \ref{sec:preq-bundle-Liet2}. This
embedding is equivariant with respect to the action of $W \rtimes
\Lat^2 $ on $\Liet^2$ and the morphism from $W \rtimes
\Lat^2 $ to the gauge group sending an element $w\in W$ to the
constant gauge transform $w$ and $(\dot p , \dot q ) \in \Lat^2$ to
$\exp ( - (x\dot p + y \dot q))$. Furthermore each gauge class of flat
connections intersects the image of the embedding.

\subsection{Prequantum bundle}  

Consider the trivial line bundle with
base $\Om^{1} ( \Si, \Lie )$ and connection $d +
\frac{1}{i} \al $, where $\al $ is the primitive of $\Om$ given by $$\al
|_a (b) = \tfrac{1}{2} \Om  (a,b ).$$ This is a prequantum bundle
and the gauge group actions lifts to it in such a way
that it preserves the trivial metric and the connection. Explicitly,
the action is given by
$$ g. ( a,u) = \Bigl( g.a,  \exp \Bigl(  -2i \pi  W(g) - i \pi \int_\Si B(
g^* \te , a ) \Bigr) u \Bigr)   $$
where $\te \in \Om^{1} (G, \Lie) $ is the left invariant Maurer-Cartan one-form and $W (g)$ is
the Wess-Zumino-Witten term
$$ W (g) =  \int _{M} \tilde{g} ^* \chi 
$$  
Here $M$ is any three-dimensional compact oriented manifold with
boundary $\Si$, $\tilde {g} \in \Ci (M, G)$ any extension of $g$ and $\chi $ is
the Cartan three-form defined in terms of the left or right-invariant
Maurer Cartan forms by 
\begin{gather*}
\chi =  \frac{1}{12} B ( [\te,
\te ] , \te) = \frac{1}{12} B ( [\bar \te,
\bar \te ] ,\bar \te)
\end{gather*}
Since $B$ is the basic inner product, the cohomology class of $\chi$ is integral.

Assume now that $\Si$ is a torus and consider the equivariant
embedding of $\Liet^2$ in $\Om^{1} ( \Si, \Lie )$ defined in section
\ref{sec:gauge-theory-pres}. By pulling back, we obtain a prequantum
bundle on $ \Liet^2$ together with an action of $\Lat^2 \rtimes W$ on it. It is not difficult to
check that this bundle and this action are exactly the ones we introduced in
section \ref{sec:symplectic-datas}.  

\subsection{Mapping class group} 

The group of orientation preserving diffeomorphisms of $\Si$ acts
symplectically on $\Om^{1} ( \Si, \Lie)$. The trivial lift to the
prequantum bundle preserves the connection and the trivial metric. After
quotienting by the gauge group, this defines an action of the mapping
class group on the moduli space of representation and its prequantum
bundle. 

When $\Si$ is a torus, we recover the action of $\Mo$ introduced in section \ref{sec:modular-group}. For any $A \in \Mo$, define the diffeomorphism of the torus
$$ \varphi_A ( x, y) = ( dx - cy , -b x + a y)$$
where $a$, $b$, $c,$ and $d$ are the coefficients of $A$. On one hand, we recover
the usual formula by considering the basis $\al = (0,-1)$ and $\be =
(1,0)$. Indeed $\varphi_A(\al) = a \al + b \be$ and $\varphi_A (\be) = c \al + d
\be$. On the other hand, 
$$ \varphi^*_A ( p dx + q dy ) = ( a  p + b q ) dx + ( cp  + d q ) dy$$
which corresponds to the action of section \ref{sec:modular-group}.

\section{Quantization} \label{sec:quantization}

Let us begin with a brief description of the general set-up. Consider a symplectic manifold $(M, \om)$ with a prequantum bundle
$L \rightarrow M$. Assume $(M,\om)$ is
endowed with a compatible positive complex structure, so $\om$ is a
$(1,1)$ form and $ - i \om (Z, \bar{Z}) >0$ for any non vanishing
tangent vector $Z$
of type $(1,0)$.  Then the prequantum bundle has a unique holomorphic
structure such that the local holomorphic sections satisfy the
Cauchy-Riemann equations: 
$$\nabla_{\bar{Z}} s  =0, \qquad \text{ for any vector
field $Z$ of type $(1,0)$.}$$ The quantum space associated to these data
is the space $H^0 (M, L^k)$ of holomorphic sections of $L^k$. It has a
natural scalar product obtained by integrating the punctual scalar
product of sections against the Liouville measure $|\om^n |/ n !$.

In a first subsection we introduce complex structures on $\mtore^2$
and define a basis of the quantum space by using theta functions. We
also describe the action of the matrices $S$ and $T$ of $\Mo$
in these basis. These results are standard. We
provide proofs in appendix. Next we we move on to a subspace of equivariant sections
with respect to the Weyl group action, the so called alternating
sections. We compute the actions of $S$ and $T$ in this space. 

\subsection{Complex structure and theta functions} \label{sec:compl-struct-theta} 
Denote by $\Hp $ the
Poincar{\'e} upper half-plane
$$\Hp = \{  x + i y / x, y \in \R , \; y >0 \}.$$  
Let $\tau \in \Hp$. Identify $\Liet^2$
with $\Liet_\C = \Liet \otimes \C$ by the isomorphism sending $(p, q)$
to $p + \tau q$. Hence $\Liet^2$ becomes a complex vector space
and $\mtore^2$ inherits a complex structure. One may compute the
symplectic form $\om$ in
terms of the complex  coordinate $ \zeta = p + \tau q$
$$ \om = \frac{ 2  \pi}{ \bar{\tau} - \tau  }  B( d \zeta, d \bar{\zeta} ).
 $$
It is a positive real form of type $(1,1)$. So the prequantum bundle $L$ has a unique
holomorphic structure compatible with the connection. We denote by $H^0_\tau
( \mtore^2 , L^k)$ its space of holomorphic sections.
 
Consider the section of $L_{\Liet^2}$ 
$$s  = \exp ( i \pi B( \zeta, q)) $$
Its covariant derivative is $
2 i \pi B( d \zeta, q) \otimes s$. Since this form is of  type  $(1,0)$, $s$ is
holomorphic. Furthermore $s$ doesn't vanish anywhere. So the holomorphic sections of $L^k$
identify with the sections over $\Liet^2$ of the form  $f s^k$ such that $f :
\Liet^2 \rightarrow \C$ is holomorphic and $f s^k$ is $\Lat
^2$-invariant. 

As previously, we embed $\Lat ^*$
in $\Liet$ via the identification given by the basic inner
product. Recall that $\Lat \subset \Lat ^*$. For any $\mu \in k^{-1} \Lat^* $, consider the theta function
 $$ \Theta_{\mu, k } ( p,q ) = \sum_{ \ga \in \mu + \Lat} \exp \bigl( 2 i \pi k \bigl(
  \tfrac{\tau}{2} B( \ga, \ga ) - B( \zeta, \ga ) \bigr) \bigr)  $$  
 This series converges uniformly on compact sets to a holomorphic
 function, it depends only on $\mu $ mod $\Lat$. 

\begin{theo} \label{theo:base_theta}
For any integer $k$, the sections $ \Theta_{\mu , k } s^k$, where $\mu$
runs over $k^{-1}\Lat ^*$ mod $\Lat$, are $\Lat^2$-invariant and form
an orthonormal basis of $H^{0}_\tau ( \mtore^2, L^k)$. Furthermore,
$$ \| \Theta_{\mu,k} s^k \| ^2 = \Bigl( \frac{2 \pi}{k} \Bigr) ^{n/2}  \Bigl( \frac{ 2i\pi
} { \tau - \bar \tau} \Bigr)^{n/2} \operatorname{Vol}( \Liet/ \Lat),$$
where $\operatorname{Vol} ( \Liet / \Lat )$ is
the Riemannian volume determined by $B$.
\end{theo}

Recall that the modular group acts on $\mtore ^2$ and its prequantum
bundle. The induced action on the sections of $L^k$
doesn't preserve the space $H^0_\tau (
\mtore ^2, L^k)$, because of the complex structure.  Actually, $A\in \Mo$ acts as a holomorphic map from $(\mtore ^2, j_\tau)$
 to $(\mtore ^2, j_{A \tau})$ with 
   $$ A \tau = \frac{ a \tau - b }{ -c \tau + d } $$
So for any $\tau$, $A$ acts as an isomorphism
\begin{gather*} \label{eq:modaction} 
   H_{\tau}  ^0 (
\mtore ^2, L^k) \rightarrow H_{ A \tau} ^0  (
\mtore ^2, L^k)
\end{gather*}
One may compute explicitly this isomorphism in the basis of theta
functions when $A$ is the matrix $S$ or $T$. We make explicit the dependence in
$\tau$ in our notations to avoid any ambiguity.

\begin{theo} \label{theo:modular-action}
For any $\tau \in \Hp$ and $\mu \in k^{-1} \Lat^*$, one has   
\begin{xalignat*}{2}
  S. \bigl( \Theta^\tau_{\mu, k } s^k_{\tau}
 \bigr)=  & C \sum_{ \mu'
    \in k^{-1} \Lat^*   \operatorname{mod} \Lat } \exp(-  2 i \pi k
B( \mu, \mu') )  \Theta^{S.\tau}_{\mu', k
} \; s^k_{S.\tau} 
\end{xalignat*}
with $C = \bigl( S.\tau / i\bigr) ^{ n/2} k^{-n/2} \operatorname{Vol} ( \Liet / \Lat )^{-1}$ and
 $$ T. \bigl(\Theta^\tau_{\mu, k } s^k_{\tau} \bigr) =  \exp ( i \pi k B( \mu, \mu )
) \Theta^{T.\tau}_{\mu  , k } \; s^k_{T.\tau}  
$$
\end{theo}
Here $(\tau /i)^{n/2}$ is the determination continuous with respect to
$\tau$ and equal
to $1$ when $\tau = i$.

\subsection{Alternating sections}\label{sec:alternating-sections}

The action of the Weyl group on $\mtore^2$  is holomorphic with respect
to the complex structure
defined by any $\tau \in \Hp$. Let us consider the alternating 
sections  of $L^k$, i.e. the sections $\Psi$ satisfying $$ w. \Psi =
(-1)^{\ell (w) } \Psi, \qquad \forall w \in W. $$
For any $\mu \in k^{-1} \Lat^*$, let 
$$ \chi_{\mu,k} = \sum_{w \in W} ( - 1) ^{\ell (w)} \Theta_{w (
  \mu),k}\;  s^k .
$$
Recall that we denote by $\Alc$ the fundamental open Weyl alcove. 
\begin{theo}  \label{theo:basis-alternating-sections}
The family $(\chi_{\mu,k}, \;  \mu \in \Alc \cap k^{-1}
\Lat^*)$ is a basis of the space of alternating holomorphic sections of $L^k$.  
\end{theo}

\begin{proof} Using that the Weyl group action preserves $\Lat$ and
  $B$, we check that for any $w \in W$, 
$$w. ( \Theta_{\mu,k}\;  s^k ) =  \Theta_{w (
  \mu),k}\;  s^k. $$ 
So $\chi_{\mu,k}$ is alternating. 

For any root $\al$ and integer $n$, the orthogonal reflexion with
respect to the hyperplane $\al^{-1}(  n)$ belongs to the affine Weyl
group $W \rtimes \Lat$. So for any $\mu \in \al^{-1} (n )$, there
exists $w \in W$ with $\ell ( w ) = 1$ such that $w ( \mu ) = \mu$
modulo $\Lat$. Hence  $ \chi_{\mu,k} = -\chi_{w(\mu),k} = - \chi
_{\mu,k}$, so $\chi_{\mu,k}$ vanishes.
  
Recall that the affine Weyl group $W \rtimes \Lat$ acts simply transitively on the set of components of $\Liet \setminus
\cup_{\al, n }  \al^{-1}(n)$ and that $\Alc$ is one of these
components. The result follows from theorem \ref{theo:base_theta}.
\end{proof}

The modular action and the action of the Weyl group on $\mtore^2$ and
$L$ commute. So the representation of the modular group preserves the
subspace of alternating sections.
\begin{theo}  \label{theo:rep_alternating-sections}
For any $\tau \in \Hp$ and $\mu \in \Alc \cap k^{-1} \Lat^*$, one has   
\begin{xalignat*}{2}
 S. \bigl( \chi^{\tau}_{\mu, k } s^k_{ \tau}
 \bigr)=  & C 
 \sum_{\substack{\mu'\in \Alc \cap k^{-1} \Lat^*, \\ w \in W }}
 (-1)^{\ell (w) }\exp(-  2 i \pi k B( \mu, w(\mu') )
 \chi^{S.\tau}_{\mu', k} \; s^k_{S.\tau}
\end{xalignat*}
with $C$ defined as in theorem \ref{theo:modular-action} and
 $$ T. \bigl(\chi^\tau_{\mu, k  } s^k_{\tau} \bigr) =  \exp ( i \pi k B( \mu, \mu )
) \chi^{T.\tau}_{\mu  , k} \; s^k_{T.\tau } 
$$
\end{theo}

\begin{proof} 
The second formula follows from theorem \ref{theo:modular-action}
using that the Weyl group acts
isometrically on $\Liet$. Let us prove the first one. By theorem \ref{theo:modular-action},
\begin{xalignat*}{2} 
   S.  ( \chi_{\mu,k}^\tau s^k_\tau ) 
= & C   \sum_{\substack{\mu ' \in  k^{-1} \Lat^* \operatorname{mod}
    \Lat, \\ w \in W }} (-1)^{\ell (w) }\exp(-  2 i \pi k B( w(\mu),
\mu') ) \; \theta_{\mu',k }^{S.\tau } s^k_{S.\tau }  \\
= & C   \sum_{\substack{\mu ' \in  k^{-1} \Lat^* \operatorname{mod}
    \Lat, \\ w \in W }} (-1)^{\ell (w) }\exp(-  2 i \pi k B( \mu,
\mu') ) \; \theta_{w(\mu'),k }^{S.\tau } s^k_{S.\tau } \\
= & C  \sum_{\mu ' \in k^{-1} \Lat^* \operatorname{mod} \Lat } \exp(-  2 i \pi k B( \mu , \mu') ) \;  \chi_{\mu '  ,k }^{S.\tau } s^k_{S.\tau } \\
= & C \sum_{\substack{\mu'\in \Alc \cap k^{-1} \Lat^*, \\ w \in W }}
\exp(-  2 i \pi k B( \mu, w(\mu') )  \;  \chi_{w(\mu ')  ,k }^{S.\tau } s^k_{S.\tau } 
\end{xalignat*}
In the last line, we used that the affine Weyl group acts simply
transitively on the set of connected components of $\Liet \setminus
\bigcup \al^{-1} (n)$ and that $\chi_\mu$ vanishes if $\mu \in
\al^{-1}(n)$. Finally, 
\begin{gather*} 
 S.  ( \chi_{\mu,k}^\tau s^k_\tau ) =  C \sum_{\substack{\mu'\in \Alc \cap k^{-1} \Lat^* \\ w \in W }}
(-1)^{\ell (w) }\exp(-  2 i \pi k B( \mu, w(\mu') )  \;  \chi_{\mu '
  ,k }^{S.\tau }  s^k_{S.\tau }
\end{gather*}
since the $\chi_{\mu,k}$'s are alternating. 
\end{proof}

\section{Geometric quantum representation} \label{sec:geom-quant-repr}

We introduce a representation of the modular group on the quantum
spaces. To do this we  identify the various spaces $H^0_\tau
(\mtore^2, L^k)$ via the sections $\Theta_{\mu,k}^{\tau} s^k_\tau $. Unfortunately
the norm of these sections and the action of the modular group depend on $\tau$ as it
appears in theorems \ref{theo:base_theta} and 
\ref{theo:modular-action}. We introduce half-form bundle to
correct this. 

\subsection{Half-form bundles}
Let us begin with some definitions. Consider a
symplectic manifold $M$ with a prequantum bundle $L$ and a positive
compatible complex structure. Then a half-form bundle is a complex
line bundle $\delta$ over $M$ with an isomorphism from $\delta^2$ to
the canonical bundle of $M$. A half-form bundle admits a natural
metric and a natural holomorphic
structure making the isomorphism with the canonical bundle a morphism of Hermitian holomorphic
bundle. The quantization of $M$ with metaplectic correction is then
the space of holomorphic sections of $L^k$ tensored with $\delta$. The scalar product is defined by integrating the punctual
norm of sections against the Liouville measure.

Let us return to our particular situation. Consider $\Om \in \wedge^n \Liet_\C^*$ such that for a basis
$(\ga_i)$ of $\Lat$, one has $$\Om ( \ga_1\wedge 
\ldots \wedge \ga_n) =  1 .$$ $\Om$  is uniquely defined up to a plus or
minus sign. For any $\tau \in \Hp$, we defined a complex structure $j_\tau$ on $\mtore^2$
via the isomorphism  
$$\mtore^2 \rightarrow \Liet_\C / (\Lat + \tau \Lat), \qquad [p,q]
\rightarrow [p + \tau q] . $$ 
So the holomorphic tangent bundle of $(\mtore^2, j_\tau)$ is naturally
isomorphic to the trivial bundle with fiber $\Liet_\C$. Consequently the
canonical bundle is naturally isomorphic to
the 
trivial bundle with fiber $\wedge^n \Liet_\C^*$. Denote by
$\Om_\tau$ the section of the canonical bundle which is sent into
the constant section equal to $\Om$ by this trivialization. 

\begin{lem}  \label{lem:norm_can_bundle}
The section $\Om_\tau$ is holomorphic and has a constant punctual norm
equal to $\bigl( \frac{ \tau - \bar \tau }{2 i \pi } \bigr)
^{n/2} \operatorname{Vol}( \Liet / \Lat ) ^{-1} $.
\end{lem}

\begin{proof}  Introduce an orthonormal basis $(u_i)$ of $\Liet$ and
denote by $(p_i)$ and $(q_i)$  the associated
linear coordinates of $\Liet^2$. Let $\zeta_\tau^i$ be the complex
coordinate $p_i + \tau q_i$. Then  if  $(\ga_i)$ is a basis of $\Lat$,
one has 
\begin{xalignat*}{2} 
 d\zeta^1_\tau \wedge \ldots \wedge d \zeta^n_\tau ( \ga_1, \ldots ,
\ga_n ) =  & dp_1 \wedge \ldots \wedge d p_n ( \ga_1, \ldots ,
\ga_n ) \\ = &\pm \operatorname{Vol} ( \Liet / \Lat )
\end{xalignat*}
where the plus or minus sign depends on the orientation of the various
basis. Consequently 
\begin{gather} \label{eq:omtau}
 \Om_\tau = \pm \operatorname{Vol} ( \Liet / \Lat ) ^{-1}
d\zeta^1_\tau \wedge \ldots \wedge d \zeta^n_\tau
\end{gather}
Now, by definition the Hermitian product of two tangent vectors $X,Y$ of type
$(1,0)$ is given by $\frac{1}{i} \om (X, \overline{Y})$. Since 
$$ \om = 2 \pi \sum dp_i \wedge d q_i = i \frac{2i \pi}{ \tau - \bar
  \tau} \sum d \zeta^i_\tau \wedge d \bar{\zeta}_\tau^i$$
the vectors $\partial_{\zeta_{\tau}^i}$ are mutually orthogonal and
$$ \bigl| \partial_{\zeta_{\tau}^i} \bigr|^2 = \frac{ 2 i \pi }{\tau -
  \overline{\tau}} .$$
So the $d \zeta^i_\tau$ are mutually orthogonal and
$$ \bigl| d \zeta_{\tau}^i \bigr|^2 = \frac{\tau -
  \overline{\tau}}{2 i \pi} $$
which implies 
$$ | d\zeta^1_\tau \wedge \ldots \wedge d \zeta^n_\tau |^2 = \Bigl(
\frac{\tau - \bar \tau}{2 i \pi} \Bigr)^{n}$$
and concludes the proof.   
\end{proof}

Let $\delta$ be a complex vector line and $\varphi$ be an
isomorphism $\delta^{\otimes 2} \rightarrow \wedge^{\operatorname{top}}
\Liet_\C^* $. Let $\omh \in \delta$ be such that $$\varphi
(\omh \otimes \omh ) = \Om.$$
For any $\tau \in \Hp$, the trivial bundle $\delta_{\tau}$  with base $\mtore^2$ and fiber $\delta$ is a
half-form bundle, with the squaring map sending $\omh$ into
$\Om_\tau$. By the previous lemma, the constant section equal to
$\Om^{1/2}$ is a holomorphic section of $\delta_\tau$ with constant
punctual norm equal to $\bigl( \frac{ \tau - \bar \tau }{2 i \pi } \bigr)
^{n/4} \operatorname{Vol}( \Liet / \Lat ) ^{-1/2} $.

Instead  of sections of $L^k$, consider now sections
of $L^k \otimes \delta_{\tau}$. The multiplication by $\omh$ is an isomorphism
$$H^{0}_\tau ( \mtore^2, L^k) \simeq H^{0}_\tau ( \mtore^2, L^k \otimes \delta_{\tau} )$$
of vector space. We deduce from theorem \ref{theo:base_theta} and lemma
\ref{lem:norm_can_bundle} the 
\begin{theo} \label{theo:norm_half_form}
For any $\tau \in \Hp$, the sections $$ \Theta_{\mu,k}^{\tau}
  s^k_\tau \otimes \omh, \qquad  \mu \in k^{-1}\Lat ^* \text{ mod }\Lat,$$ form a basis of
  $H^{0}_\tau ( \mtore^2, L^k \otimes \delta_{\tau} )$. They are mutually
  orthogonal and  
$$ \| \Theta_{\mu,k}^{\tau} s^k_\tau \otimes \omh \| ^2 = \Bigl( \frac{2 \pi}{k} \Bigr) ^{n/2} . $$
\end{theo}

For any  $\tau_1$ and $\tau_2$ in
$\Hp$, let $\Psi_{\tau_1, \tau_2,k}$ be the isomorphism from
$H^{0}_{\tau_1} ( \mtore^2, L^k \otimes \delta_{\tau_1} )$ to $H^{0}_{\tau_2} ( \mtore^2, L^k \otimes \delta_{\tau_2}
)$ defined by 
$$ \Psi _{\tau_1, \tau_2,k} ( \Theta_{\mu,k}^ {\tau_1}  s^k_{\tau_1}
\otimes \omh ) = \Theta_{\mu,k}^{ \tau_2} s^k_{\tau_2} \otimes \omh,
\qquad \forall \mu .$$
By the previous theorem, $ \Psi _{\tau_1, \tau_2,k}$
is a unitary map. 

\subsection{Modular action} \label{sec:modular-action}

Let $A \in \Mo$ be the matrix with coefficients $a$, $b$, $c$ and $d$. The map $\varphi_A ( p,q )= ( ap +b q, cp + dq)$ is a
holomorphic map from $(\mtore^2 , j_\tau)$ to $(\mtore^2, j_{A. \tau})$. Using the
coordinates introduced in the proof of lemma \ref{lem:norm_can_bundle}, we show that 
\begin{gather} \label{eq:1}
 \varphi_{A} ^* \Om_{ A.\tau} = ( - c \tau + d )^{-n} \Om_{\tau}.
\end{gather}
We will lift the action of $A$ to the half-form bundles in such a way
that it squares to $(\varphi_A^*)^{-1}$. Since $(- c \tau +d)^n$ doesn't
admit a preferred square root, we have to pass to an extension of the
modular group. 

Let $\Mo_2$ be the set of  pairs $(A, e)$ where $A \in \Mo$
and $ e$ is a continuous function from $\Hp$ to $\C$ satisfying
$$e(\tau)^2 =
( - c\tau + d)^n \quad \text{ if } \quad A =   \left(
  \begin{array}{cc} a & b   \\  c  &  d  \end{array} \right)$$
$\Mo_2$ is a group with the product given by $(A,e). ( A', e') = ( A A', e'')$ where $e'' (
\tau ) = e ( A' \tau ) e' ( \tau)$.  
$\Mo_2$ acts on the product $\Hp \times \mtore^2 \times \delta$
$$ ( A, e) . \bigl( \tau, p, q, z \omh \bigr) = \Bigl( \frac{ a \tau - b }{
-  c \tau + d} ,  ap + b q, cp + dq , z e( \tau) \omh \Bigr)$$
Restricting to a particular value $\tau$, we obtain a morphism from
$\delta_\tau$ to $\delta_{A.\tau}$ lifting the action of $A$ on $\mtore
^2$. The important point is that
the square of this isomorphism is the natural map between the canonical bundles of
$(\mtore^2, j_{\tau})$ and $( \mtore^2 , j_{A \tau} )$. 
Since the Hermitian and holomorphic structures of a half-form bundle
are determined by the corresponding canonical bundle, the morphism
$\delta_{\tau} \rightarrow \delta_{A.\tau}$ that we consider is a unitary holomorphic isomorphism. 

As previously we lift trivially the action of $\Mo$ to the prequantum bundle. Then we obtain for any $(A,e, \tau) \in
\Mo_2 \times \Hp$ an isomorphism
$$   \varphi_{(A,e)^{-1},\tau}^* : H_{\tau}  ^0 (
\mtore ^2, L^k \otimes \delta_{\tau }) \rightarrow H_{A.\tau} ^0  (
\mtore ^2, L^k \otimes \delta_{A.\tau} ) $$ 
By trivial reason, it is a unitary map.  
Miraculously, these isomorphisms  all fit together.

\begin{prop} For any $\tau_1
 , \tau_2 \in \Hp$ and any $(A,e) \in \Mo_2$, the diagram 
$$ \begin{CD} H_{\tau_1}  ^0 (
\mtore ^2, L^k \otimes \delta_{\tau_1 }) @>
 \varphi_{(A,e)^{-1},\tau}^*   >>   H_{A.\tau_1} ^0  (
\mtore ^2, L^k \otimes \delta_{A.\tau_1} ) \\ 
 @VV\Psi_{\tau_1, \tau_2}V  @VV\Psi_{A.\tau_1, A.\tau_2}V \\
H_{\tau_2}  ^0 (
\mtore ^2, L^k \otimes \delta_{\tau_2 }) @>
\varphi_{(A,e)^{-1},\tau}^* >>   H_{A.\tau_2} ^0  (
\mtore ^2, L^k \otimes \delta_{A.\tau_2} )
\end{CD} $$
commute. 
\end{prop}

\begin{proof} It is sufficient to prove it for $(A,e)= (T,1)$, $(S,e)$
  or $(\id , -1)$ since
  these elements generate $\Mo$. The actions of $(T,1)$ and
  $(S,e)$ in the basis $ \Theta^\tau_{\mu, k } s^k_{\tau} \omh$ are
  given by the same formulas as in theorem \ref{theo:modular-action}
  except that the constant $C$ has to be replaced by 
$$ C' = C e( \tau) =  e ( i ) k^{-n/2} \operatorname{Vol} ( \Liet / \Lat )^{-1}
$$
This proves the result because $C'$ does not depend on $\tau$.
\end{proof}

\begin{cor}\label{cor:def_R2}
For any $\tau$, the map 
$$ R_2: (A,e) \rightarrow  \varphi_{(A,e)^{-1},\tau}^* \circ \Psi_{\tau,
  A^{-1}\tau}$$
 is a unitary representation of $\Mo_2$ on $ H_{\tau}  ^0 (
\mtore ^2, L^k \otimes \delta_{\tau })$. 
\end{cor}

Consider now the diagonal action of the Weyl group on $\mtore^2$. Its
trivial lift to the half-form bundle acts holomorphically and
preserves the metric. So we have a unitary representation of $W$ on $ H_{\tau}  ^0 (
\mtore ^2, L^k \otimes \delta_{\tau })$. It is easy to see that this
representation commutes with the one of $\Mo_2$. This gives a
representation of $\Mo_2$ on the subspace of
alternating sections that we denote by $R_2^{\alt}$. By theorem
\ref{theo:basis-alternating-sections}, we have

\begin{theo} \label{theo:matrice_R}
The coefficients of the matrix of $R_2^{\alt}
(S,e) $ and $R_2^{\alt} ( T,1)$ in the basis $  \chi_{\mu,k} \otimes \omh$, $\mu \in \Alc \cap k^{-1}
 \Lat^* $ are respectively 
$$  \frac{e(i)^{-1}}{k ^{\frac{n}{2}}  \operatorname{Vol} ( \Liet / \Lat )}   \sum_{w \in W } (-1)^{\ell (w) }\exp(-  2 i \pi k B( \mu, w(\mu') )  ) $$
 and
  $$  \delta_{\mu,\mu'}  \exp ( i \pi k B( \mu, \mu ) ).$$
\end{theo}

\section{Algebraic  projective representation}  \label{sec:modul-tens-categ}

 Consider the category of representations of the quantum group $U_q (
 \Lie)$ for $q$ being the root of unity.  This category has a
 subquotient, called the fusion category, which is a modular tensor
 category. Following \cite{Tu}, we can define a natural projective representation of the mapping
 class group of any 2-dimensional surface with marked points on
 appropriate spaces of morphisms in this category. In particular, for
 the torus, we get a projective representation of the modular group that we
 define in this section.

\subsection{The representation $R_{\infty}$}  \label{sec:rep_infty}

Denote by $\rho$ the half sum of positive roots and by $ h^\vee = 1 + \rho (
\al_0 ^\vee)$ the dual Coxeter number. Let $k$ be an integer bigger
than $h^\vee$. Then the set of admissible weights at level $k - h
^{\vee}$ is 
$$ C_k : = \Lat^* \cap ( k - h^\vee ) \overline{\Alc}$$ 
For any $\la, \mu
\in C_k$, let 
$$ \tilde s _{\la \mu} := \Bigl| \frac{ \Lat^* }{k \Lat } \Bigr| ^{- 1/ 2} i
^{|\Delta_+|} \sum_{w \in W } (-1) ^{\ell ( w) } \exp \Bigl( - \frac{2 i
\pi } {k} B( w ( \la + \rho ), \mu + \rho ) \Bigr)$$ 
and 
$$ \tilde t_{\la \mu} := \delta_{\la \mu }  \exp \Bigl( \frac{ i
\pi } {k} B( \la , \la + 2 \rho ) \Bigr). $$
% Recall that the modular group $\Mo := \Sl ( 2, \Z)$ is generated by 
% $$ S =  \left(
%   \begin{array}{cc} 0 & -1   \\  1  &  0   \end{array} \right) ,
% \qquad T =  \left(
%   \begin{array}{cc} 1 & 1  \\  0 & 1 \end{array} \right)  $$
% The map sending $S$ and $T$ to the matrices $(s_{\la \mu})$ and
% $(t_{\la \mu })$ extends to a projective representation of $\Mo$. Actually we get a true representation if we replace $(t_{\la
%     \mu } )$ by $  \exp (-  \frac{2 \pi i c }{ 24}) (t_{\la \mu })$
%   with  $c$ being the central charge
% $$ c = ( k - h ^\vee ) \dim \Lie / k .$$
% This correction is rather ad hoc and doesn't have a higher genus
% generalization. Instead we may consider a representation of a
% central extension of $\Mo$. This has the further advantage
% that we will be able to understand the relation with the geometric
% representation of $\Mo$.

Consider the action of the modular group on the real projective line
induced by the standard action on $\R^2$. Choose a base
point $L_o \in \proj^1( \R)$. Let $\Mo_\infty$
be the set of pairs $(A,\gamma)$
where $A \in \Mo$ and $\ga$ is a homotopy
class (with fixed endpoint) of a path in $\proj^1 ( \R)$ from $A. L_o$
to $L_o$. $\Mo_\infty$ is an extension by $\Z$ of the modular group, the
product being given by 
$$ (A, \ga).( A' , \ga')= ( A A' , A( \ga') \star \ga ) .$$
Assume that $L_o = [1,0]$. Then $\Mo_\infty $ is generated by $\hat{s} = ( S, [\phi])$, $
\hat{t} = (T, [1,0])$ and $\hat \gamma = (\id, [\ga])$ where $\phi$ and $\ga$ are
the paths 
$$ \phi( t) =  [ \sin ( t \frac{\pi}{2} ),\cos (t \frac{\pi}{2})],
\quad \ga( t) = [ \cos (t \pi) , \sin ( t \pi) ] ,$$  
with $0 \leqslant  t \leqslant 1$.
\begin{theo} \label{theo:proj-rep-mod}
$\Mo_\infty$ admits a representation $R_\infty$ on $\C^{C_k}$
determined by 
\begin{gather*} 
  R_\infty (\hat s ) =  (\exp (- i \tfrac{\pi}{4}c ) \tilde s_{\la
  \mu} ) ,  \quad R_\infty (\hat t ) = ( \tilde t_{\la \mu }), \quad 
   R_\infty (\hat \gamma ) =  ( \exp ( i \tfrac{\pi}{2}c)
\delta_{\la \mu} )   .
\end{gather*}
\end{theo}

The reader is referred to the book of
Bakalov-Kirillov \cite{BaKi} for a detailed exposition on modular tensor
category and the fusion category of $U_q ( \Lie)$.  The formulas for
$s_{\la \mu } $ and $t_{\la \mu }$ are given in theorem 3.3.20, that
they give a projective representation is the content of chapter 3.1,
cf. remark 3.1.9. The definition of the central extension is in
chapter 5.7, the case of the torus being treated in example 5.7.7.
Note also there is a misprint in formula for $s_{\la \mu }$, compare with
proposition 3.8 of \cite{Ki}. 

\subsection{Comparison of the representations $R_2^{\alt}$ and $R_\infty$}

To compare $R_2^{\alt}$ with $R_\infty$, we introduce a third extension
$\Mo_4$ of the modular group. By definition, $\Mo_4$ consists of the
pairs  $(A, e)$ where $A \in \Mo$
and $ e$ is a continuous function from $\Hp$ to $\C$ satisfying
$e(\tau)^4 =
( - c\tau + d)^{2n}$, with $a$, $b$, $c$, $d$ the coefficients of
$A$. The product is given by the same formula as for $\Mo_2$. Let
$\Z_4 = \{ \pm 1 , \pm i \}$. The
morphism
$$ \Mo_2 \times \Z_4 \rightarrow \Mo_4, \qquad (A,e , u) \rightarrow
(A, eu ) $$
is onto with kernel $\{ ( \id, 1, 1), ( \id, -1 , -1 ) \} $. Thus we
have a representation $R_4$ of $\Mo_4$ given by 
$$ R_4 ( A,eu) = u R_2^{\alt} (A,e), \qquad (A,e) \in \Mo_2 , \; u \in \Z_4 .$$  
Recall that $\Mo_{\infty}$ is generated by $\hat s $, $ \hat t $ and $
\hat \ga $.
\begin{lem} 
There exists a morphism $\Psi$ from $\Mo_\infty$ onto $\Mo_4$ such that
\begin{gather*} 
 \Psi ( \hat s ) = (S,e) \quad \text{ with } \quad e( i ) = e^{-i
  \frac{\pi}{4} n} \\
\Psi ( \hat \ga ) = ( \id , i ^{\dim G} ) ,\qquad \Psi ( \hat t )  = (
T, 1).
\end{gather*}
\end{lem}

\begin{proof} 
Consider the groups $\Mo_\infty^\R$ and $\Mo_4^\R$ defined
exactly as $\Mo_{\infty}$ and $\Mo_4$ except that we replace the
modular group by $\Mo^\R:=\Sl ( 2, \R)$. We will construct a morphism from
$\Mo_\infty^\R$ onto $\Mo_4^\R$ whose restriction to
$\Mo_\infty$ is $\Psi$. Let $j $ be the morphism from the unit circle
$\U \subset \C$ into $\Mo^\R$  
$$j ( u ) = \left( \begin{array}{cc} \operatorname{reel} u  & -\operatorname{im} u \\ \operatorname{im} u & \operatorname{reel} u
\end{array} \right) 
$$
Since $\Mo^\R$ admits a deformation retract onto $j ( \U)$, one
has a bijective correspondence between the morphisms from  $\Mo_\infty^\R$ to $\Mo_4^\R$
covering the identity of $\Mo^\R$ and the morphisms from $j^*
\Mo_\infty^\R$ to $j^* \Mo_4^\R$ covering the identity of
$\U$.  On one hand 
$$ j^* \Mo_4^\R \simeq \{ ( u, e) \in \U \times \C^* / \; e^4 = u
^{-2n} \} =: \U^4$$
because for $\tau = i $, $-c \tau + d = u^{-1}$ if $a$, $b$, $c$, and
$d$ are the coefficients of $ j (u)$. 
On the other hand
$$ j^* \Mo_\infty^\R \simeq \{ ( u, \theta) /  \; \exp( i \theta
) = u^2 \} =: \U^\infty $$
where we send $( u, \theta)$ into $( j(u),  [\ga])$ with $\ga$ the path
$$  \ga (t) =  [ \cos ( \tfrac{1-t}{2} \theta), \sin (
\tfrac{1-t}{2} \theta  ) ] ,  \qquad t \in [0,1]. $$ 
In particular $ \hat s $ and $ \hat \gamma $ are identified with $(
i , \pi)$ and $( 1 , -2\pi)$. The morphism we are looking for is  
$$ \U^\infty \rightarrow \U^4, \qquad ( u , \theta) \rightarrow ( u ,
u ^{|\Delta_+|} e^{-i \frac{\theta}{4} \dim G} )$$
Here $|\Delta_+|$ is the number of positive roots of $G$. The images of
$\hat s $ and $\hat \gamma $ are given by a straightforward
computation using that $\dim G =
2 |\Delta_+ | + n$.   
\end{proof}

\begin{lem} For any $k$, there is a morphism $\zeta_k : \Mo_\infty \rightarrow \C^*$
  such that
$$ \zeta_k ( \hat s ) = e^{i \frac{\pi}{4} x_k}, \quad  \zeta_k ( \hat
t ) = e^{-i \frac{\pi}{12} x_k}, \quad \zeta_k ( \hat \ga ) = e^{ -
  i \frac{\pi}{2} x_k}
$$
with $x_k = h^\vee \dim G / k $ where $h^\vee$ is the dual Coxeter
number of $G$.  

\end{lem}

\begin{proof} The value of $x_k$ does not matter for the proof. We
  only use that $\Mo_\infty$ is generated by $\hat{s}$, $
  \hat{t} $ and $\hat{\ga}$ with the relations 
$$ \bigl( \hat{\gamma} \hat{s}^2 \bigr) ^2 = \id , \quad \bigl( \hat{s} \hat{t} \bigr) ^3 =
\hat{s}^2 , \quad \hat{\gamma} \text{ and } \hat{s}^2 \text{ are central}$$  
as it is asserted in \cite{BaKi}, example 5.7.7.
\end{proof}

So by the previous lemmas, we define a representation of $\Mo_\infty$
by $$ \tilde{R}_\infty (A,\ga)  =  \zeta_k (A,\ga) . R_4 ( \Psi ( A, \ga ))$$
We will prove it is the same as $R_\infty$. To do this we have to
identify $\C^{C_k}$ with  $ H_{\tau}  ^0 (
\mtore ^2, L^k \otimes \delta_{\tau })$. We need the
following well-known fact. 
\begin{lem}
For any $k$, the map sending $\la$ into $\frac{\la + \rho }{k}$ is a
bijection between $C_k$ and $\Alc \cap k^{-1}
 \Lat^* $
\end{lem}
The identification is then defined by sending the canonical basis  of
$\C^{C_k}$ into the basis $  \chi_{\mu,k} \otimes \omh$, $\mu \in \Alc \cap k^{-1}
 \Lat^* $.

\begin{theo} 
The matrices of $\tilde{R}_\infty (\hat s ) $,  $\tilde R_\infty (
\hat t )$, $\tilde R_\infty (\hat \ga )$ in the basis $\chi_{ (\la +
    \rho)/k  ,k} \otimes \omh$, $\la \in
C_k$ are respectively given by 
\begin{gather*} 
  (\exp (- i \tfrac{\pi}{4}c ) \tilde s_{\la
  \mu} ) ,  \qquad  ( \tilde t_{\la \mu }) , \qquad 
    ( \exp ( i \tfrac{\pi}{2}c)
\delta_{\la \mu} )   .
\end{gather*}
where $(s_{\la \mu })$ and $(t_{\la  \mu})$ are the matrices
defined in section \ref{sec:modul-tens-categ}. 
\end{theo}

Not only does it prove that $\tilde R_\infty = R_\infty$, but
  it also proves theorem \ref{theo:proj-rep-mod}. 
\begin{proof}
This follows from theorem \ref{theo:matrice_R}.  Since $\Psi ( \hat{s} ) =
(S,e)$ with $e( i ) = e^{-i \frac{\pi}{4}}$ and $\zeta_k (\hat{s} ) =
\exp ( i \frac{\pi}{4}  \dim G \frac{h^\vee}{k}  )$, the matrix of $\tilde{R}_\infty (\hat s ) $ is 
\begin{xalignat*}{2} 
  &  \frac{e^{i \frac{\pi}{4} \bigl( n + \frac{h
        ^\vee}{k} \dim G \bigr) } }{k ^{\frac{n}{2}}
  \operatorname{Vol} ( \Liet / \Lat )}   \sum_{w \in W } (-1)^{\ell
  (w) }\exp(-  2 i \pi k  B( \tfrac{\rho +  \mu}{k}, w(
\tfrac{\rho + \la}{k}) )   \\
 = &   e^{i \frac{\pi}{4} \bigl( n + \frac{h
        ^\vee}{k} \dim G \bigr) }   \Bigl| \frac{ \Lat^* }{k \Lat } \Bigr| ^{-1/ 2} \sum_{w \in W } (-1)^{\ell
  (w) }\exp(-  2 i \pi k^{-1}  B(\rho +  \mu, w(\rho + \la) ) 
\end{xalignat*}
because
$$ \Bigl| \frac{ \Lat^* }{\Lat } \Bigr| = \frac{\operatorname{Vol} ( \Liet/
\Lat ) }{ \operatorname{Vol} ( \Liet/
\Lat ^*  )} =  \operatorname{Vol} ^2 ( \Liet/
\Lat ).$$ 
To conclude it suffices to compare with the formula defining   $\tilde s_{\la \mu}$. Since $\Psi ( \hat t ) = ( T, 1)$ and because of the value of $\zeta_k (\hat{t} )$,   the matrix of  $\tilde R_\infty (
\hat t )$ is 
\begin{xalignat*}{2} 
 &  e^{-i \frac{\pi}{12}  \dim G \frac{h^\vee}{k}  }   \delta_{\mu,\la}
\exp ( i \pi k B( \tfrac{\rho + \mu }{k} , \tfrac{ \rho + \mu}{k} )  )
\\
= &  \delta_{\mu,\la}
\exp ( i \pi k^{-1}  B(  \mu  , 2 \rho + \mu  ))
\end{xalignat*}
where we have used that
$$ \frac{B( \rho, \rho )}{2 k } = \frac{1}{24} ( \dim G - c )$$ which
follows from Freudenthal's strange formula. 
\end{proof}

\section{Complex structure dependence in the semiclassical limit} \label{sec:compl-struct-depend}

\subsection{Fourier integral operators} \label{sec:def_FIO_1}

Let $M$ be a symplectic compact manifold with a prequantization
bundle $L$. Consider two pairs $(j_a, \delta_a)$ and $(j_b ,
\delta_b)$ consisting of a positive complex structure of $M$ 
together with a half-form bundle. 
We say that a section
$\varphi$ of $\Hom (\delta_a, \delta_b)$ is a half-form bundle
isomorphism if its square at $x$ is given by 
$$ \varphi_x ^{\otimes 2} =  \pi_{b,a,x}^* :
\wedge^{\operatorname{top}} E^*_{a,x} \rightarrow
\wedge^{\operatorname{top}} E_{b,x} ^* $$
where $\pi_{b,a,x}$ is the projection from $E_{b,x} := T^{1,0}_x (M, j_b)$ to
  $E_{a,x} := T^{1,0}_x (M, j_a)$ with kernel $\overline E_{b,x} $.

 For $c =a, b$, denote by $\Hilb_k (c)$ the space of sections of
 $L^k \otimes \delta_c$ holomorphic with respect to $j_c$. 
Consider a sequence
$(S_k)$ such that for every $k$, $S_k$ is an operator $\Hilb _k
(a) \rightarrow \Hilb _k (b)$.  The scalar product of $\Hilb _k
(a)$ gives us an isomorphism 
$$ \Hom (\Hilb _k (a) , \Hilb _k (b) ) \simeq  \Hilb _k (b)
\otimes \overline{\Hilb}
_k (a) .$$
The latter space can be regarded as the space of holomorphic sections of  
$$ (L^k \otimes \delta_b ) \boxtimes (\bar{L}^k \otimes \overline{\delta}_a)
\rightarrow M^2,$$
where $M^2$ is endowed with the complex structure $(j_b, -j_a)$. The
section $S_k (x,y)$ associated in this way to $S_k$ is its Schwartz kernel. 

We
say that $(S_k)$
is a Fourier integral operator with symbol the half-form bundle isomorphism $\varphi$ if 
\begin{gather} \label{def:FIO}
 S_k(x,y) =  \Bigl( \frac{k}{2\pi} \Bigr)^{n} E^k(x,y) f(x,y,k) + O
(k^{-\infty}) \end{gather}
where 
\begin{itemize}
\item 
$E$ is a section of  $L \boxtimes \bar{L}
\rightarrow M^2$ such that  $\| E(x,y)
\| <1 $ if $x \neq y$, 
$$ E (x,x) = u \otimes \bar{u}, \quad \forall u \in L_x \text{ such that }
  \| u \| = 1, $$
and $ \bar{\partial} E \equiv 0 $
modulo a section vanishing to any order along the diagonal.
\item
  $f(.,k)$ is a sequence of sections of  $ \delta_b  \boxtimes  \bar{\delta}_a
\rightarrow M^2$ which admits an asymptotic expansion in the $\Ci$
  topology of the form 
$$ f(.,k) = f_0 + k^{-1} f_1 + k^{-2} f_2 + ...$$
whose coefficients satisfy $\bar{\partial} f_i \equiv 0  $
modulo a section vanishing to any order along the
diagonal. 
\item The restriction to the diagonal of the leading
coefficient $f_0$ is equal to $\varphi$, if we identify $\delta_b
\otimes \overline{\delta}_a$ with $\Hom ( \delta_ a , \delta_b)$ using
the metric of $\delta_a$. 
\end{itemize}

\subsection{The maps $\Psi_{\tau_1, \tau_2, k}$ in the semi-classical limit}  

Recall that for any $\tau \in \Hp$, we defined a complex structure on
$\mtore^2$ together with a half-form bundle
$\delta_{\tau}$. Consider the morphism $\varphi_{\tau_1,
  \tau_2}$ form $\delta_{\tau_1}$ to $\delta_{\tau_2}$ given by  
$$ \varphi_{\tau_1, \tau_2} (\Om^{1/2} )  = \Bigl( \frac{ \tau_1 -
  \overline{\tau}_1}{\tau_2 - \overline{\tau}_1} \Bigr) ^{n/2}
\Om^{1/2}  $$
where the square root is determined in such a way that it depends
continuously on $\tau_1$, $\tau_2$ and equal to $1$ when $\tau_1 =
\tau_2$. 

\begin{lem} \label{lem:half_form_morphism}
The map $\varphi_{\tau_1,
  \tau_2}$ is a half-form bundle isomorphism.
\end{lem}
\begin{proof}  Introduce an orthonormal basis $(u_i)$ of $\Liet$ and
denote by $(p_i)$ and $(q_i)$  the associated
linear coordinates of $\Liet^2$. Let $\zeta_\tau^i$ be the complex
coordinate $p_i + \tau q_i$. Let $E_\tau =\operatorname{span} (
\partial_{\zeta_\tau^1} , \ldots ,  \partial_{\zeta_\tau^n})$ be the
space of vector of type $(1,0)$ with respect to the complex structure
$j_\tau$. Let $\pi_{\tau_2, \tau_1}$ be the projection from
$E_{\tau_2}$ to $E_{\tau_1}$ with kernel $\overline{E}_{\tau_2}$. 
One has 
$$ d \zeta_{\tau_2}^i = \Bigl( \frac{\tau_2 - \overline{\tau}_1}{ \tau_1 -
  \overline{\tau}_1} \Bigr) d \zeta_{\tau_1}^i +  \Bigl( \frac{ \tau_1 - \tau_2}{ \tau_1 -
  \overline{\tau}_1} \Bigr) d \overline{\zeta}_{\tau_1}^i
$$
which implies that 
$$ \pi_{\tau_2, \tau_1} ^* d \zeta_{\tau_1}^i = \Bigl( \frac{ \tau_1 -
  \overline{\tau}_1}{\tau_2 - \overline{\tau}_1} \Bigr) d
\zeta_{\tau_2}^i $$
and then  
$$ \pi_{\tau_2, \tau_1} ^* ( d \zeta_{\tau_1}^1 \wedge \ldots \wedge d
\zeta_{\tau_1}^n ) = \Bigl( \frac{ \tau_1 -
  \overline{\tau}_1}{\tau_2 - \overline{\tau}_1} \Bigr) ^{n}  d \zeta_{\tau_2}^1 \wedge \ldots \wedge d
\zeta_{\tau_2}^n  $$
Finally recall that by equation (\ref{eq:omtau}) the squaring map of $\delta_\tau$ sends
$\Om^{1/2}$ into $ \pm \operatorname{Vol} ( \Liet / \Lat )d \zeta_{\tau}^1 \wedge \ldots \wedge d
\zeta_{\tau}^n  $. 
\end{proof}

\begin{theo} \label{theo:FIO}
For any $\tau_1, \tau_2 \in \Hp$, the sequence $(
  \Psi_{\tau_1, \tau_2, k })_k$ is a Fourier integral operator with
  symbol $\varphi_{\tau_1, \tau_2}$. 
\end{theo}

Section \ref{sec:proof-theo} is devoted to the proof.

\subsection{Proof of theorem \ref{theo:FIO}} \label{sec:proof-theo}

As previously we lift everything from $\mtore^2$ to $\Liet^2$. Denote
by $p_2, q_2$ (resp. $p_1, q_1$) the coordinates on the left
(resp. right) factor of $\Liet^2 \times \Liet^2$. Let $\zeta_i = p_i +
\tau_i q_i$ for $i =1,2$. 

Let us first compute  the section $E$ and the leading coefficient $f_0$ of (\ref{def:FIO}). Let $L_{\Liet^2}$ be the trivial line bundle over $\Liet^2$ with the
connection defined in section \ref{sec:symplectic-datas}. Consider the bundle $L
_{\Liet^2}\boxtimes \overline{L}_{\Liet^2}$ endowed with the holomorphic
structure compatible with the connection and the complex
structure $(j_{\tau_2}, - j_{\tau_1})$ of $\Liet^2 \times \Liet^2$. 

\begin{lem} \label{lem:E_et_f0}
The section of $L
_{\Liet^2}\boxtimes \overline{L}_{\Liet^2}$ 
$$\exp\Bigl(  \frac{ - i \pi }{ \tau_2 - \overline{\tau}_1 }B(\zeta_{2} -
\overline{\zeta}_{1}, \zeta_{2} -
\overline{\zeta}_{1})  \Bigr)  s_{\tau_2} \boxtimes \overline{s}_{\tau_1} $$ 
is holomorphic and restricts to the constant section equal to 1 on the
diagonal.   The morphism $\varphi_{\tau_1, \tau_2}$ is sent to the
constant section equal to 
$$  \Bigl( \frac{2 i \pi }{ \tau_2 - \overline{\tau_1} } \Bigr)^{n/2}
\operatorname{Vol}( \Liet/ \Lat) \Om^{1/2} \otimes \overline{\Om}^{1/2} $$
by the isomorphism between $\Hom ( \delta_{\tau_1}, \delta_{\tau_2} )$
and $ \delta_{\tau_2} \otimes \overline{\delta}_{\tau_1}$ induced by
the metric of $\delta_{\tau_1}$.
\end{lem}

The first part is proved by a straightforward computation and the
second part follows from lemma \ref{lem:norm_can_bundle}. By theorem \ref{theo:norm_half_form}, 
the Schwartz kernel of $\Psi_{\tau_1,\tau_2,k}$ lifts from $\mtore^2
\times \mtore^2$ to $\Liet^2
\times \Liet^2$ into  
\begin{xalignat*}{2}
  S_k( p_2, q_2, p_1, q_1)=  &   \sum _{\mu \in (k^{-1}\Lat ^*)/\Lat} \Theta_{\mu,k}^{\tau_2} ( p_2, q_2)
\overline{\Theta}^{\tau_1}_{\mu, k } ( p_1, q_1)  (s_{\tau_2}^k \Om^{1/2}) \boxtimes
(\overline{s}^k_{\tau_1} \overline{\Om}^{1/2}) \\
= &  \sum _{\substack{\mu \in (k^{-1}\Lat ^*) / \Lat \\ \ga_1 , \ga_2
    \in \mu + \Lat }} f_{\ga_2,\ga_1} ( p_2, q_2, p_1, q_1)   (s_{\tau_2}^k \Om^{1/2}) \boxtimes
(\overline{s}^k_{\tau_1} \overline{\Om}^{1/2}) 
\end{xalignat*}
where the coefficients are given by
 $$ f_{\ga_2,\ga_1} = \exp \bigl( 2 i \pi k \bigl(
  \tfrac{\tau_2}{2} B( \ga_2, \ga_2 ) - B( \zeta_2, \ga_2 ) - \tfrac{\overline{\tau}_1}{2} B( \ga_1, \ga_1 ) + B( \overline{\zeta}_1, \ga_1 ) \bigr) \bigr).
$$

\begin{lem} \label{lem:negl}
For any compact set $K$ of $\Liet$ such that $K \cap  \Lat =
\emptyset$, there exists $C>0$ such that for all $p_2,q_2, p_1,q_1$
satisfying $q_1 - q_2 \notin K$ one has 
$$  \bigl| S_k( p_2, q_2, p_1, q_1)  \bigr| \leqslant C e^{-k/C}$$
\end{lem}

This shows that the sequence of Schwartz kernels of $\Psi_{\tau_1,
  \tau_2, k}$ is a $O(k^{-\infty})$ outside the
diagonal. 
For the proof of the lemma we need the following general estimates. 

\begin{lem} \label{lem:estim}
For any $C>0$ there exists $C'>0$ such that
$$ \sum_{\ga \in \Z^n}  e^{-k C |\ga - x |^2} \leqslant C' k^{n/2},
\qquad \forall x \in \R^n $$
For any $C>0$, for any compact $K$ of $\R^n$ and for any subset $P$
of $\Z^n$ such that $ K \cap ( \Z^n \setminus P) = \emptyset$, there
exists $C'>0$ such that
$$ \sum_{\ga \in \Z^n \setminus P }  e^{-k C |\ga - x |^2} \leqslant C'
  e^{- k /C'}, \quad \forall x \in K.  $$
\end{lem}

\begin{proof}[Proof of lemma \ref{lem:negl}]
By a straightforward computation we obtain that
$$  \bigl| f_{\ga_2,\ga_1} s_{\tau_2}^k \boxtimes
\overline{s}^k_{\tau_1} \bigr| =   \exp \Bigl( - k \pi \Bigl( \frac{\tau_2 - \overline{\tau}_2}{2i} | q_2 - \ga_2 |^2  +
\frac{\tau_1 - \overline{\tau}_1}{2i}  | q_1 - \ga_1 |^2 \Bigr)
\Bigr) $$
Hence for some positive $C$, 
$$  \bigl| f_{\ga_2,\ga_1} s_{\tau_2}^k \boxtimes
\overline{s}^k_{\tau_1} \bigr|  \leqslant e ^{-k C ( | \ga_2 - q_2
|^2 + |  ( q_2 - q_1 ) - (\ga_2 - \ga_1) |^2  )}   $$
So with $q = q_2 - q_1$,  
\begin{xalignat*}{2} 
 \sum_{\ga_1 , \ga_2 \in \mu + \Lat }\bigl| f_{\ga_2,\ga_1} s_{\tau_2}^k \boxtimes
\overline{s}^k_{\tau_1} \bigr| \leqslant  &  \sum_{ \substack{ \ga_2 \in
    \mu + \Lat \\ \ga \in \Lat }}    e ^{ -k C ( | \ga_2 - q_2
|^2 + | q  - \ga |^2  )} \\ 
 \leqslant & \Bigl( \sum_ {\ga \in
    \Lat } e^{ -k C | \mu + \ga - q_2
|^2 }\Bigr) \Bigl( \sum_{\ga \in \Lat } e^{ -k C | \ga - q
|^2 } \Bigr)
\end{xalignat*}
By lemma \ref{lem:estim}, there exists $C'>0$  such that the first factor is bounded by
$C'k^{n/2}$, the second one by  $C' e^{-k/C'}$, and these estimates are uniform with respect to $\mu$,
$q_1$ and $q_2$ such that $q_1 - q_2 \in K$. Since the cardinal of $(k^{-1} \Lat^*) / \Lat$ is
$k^{n} |\Lat^* / \Lat|$, we obtain with a larger $C'$ that   
$$ \sum_{\substack{\mu \in (k^{-1}\Lat ^*) / \Lat \\ \ga_1 , \ga_2
    \in \mu + \Lat }} \bigl| f_{\ga_2,\ga_1}( p_2, q_2, p_1, q_1 ) s_{\tau_2}^k \boxtimes
\overline{s}^k_{\tau_1} \bigr| \leqslant C' e^{-k/C'}$$
which proves the result.
\end{proof}
Using exactly the same method and lemma \ref{lem:estim} with $P
=\{0\}$, we show 
\begin{lem} 
There exists $C>0$ such that 
$$ \sum_{\substack{\mu \in (k^{-1}\Lat ^*) / \Lat \\ \ga_1 , \ga_2
    \in \mu + \Lat, \; \ga_1 \neq \ga_2 }} \bigl| f_{\ga_2,\ga_1}( p_2, q_2, p_1, q_1 ) s_{\tau_2}^k \boxtimes
\overline{s}^k_{\tau_1} \bigr| \leqslant C e^{-k/C}$$
for all $p_2$, $q_2$, $p_1$ and $q_1$ such that $|q_1 - q_2 |
\leqslant R/2$ with $R = \min \bigl\{ |\ga|, \; \ga \in \Lat \setminus \{ 0 \} \bigr\}$.
\end{lem}

So up to a $O(k^{-\infty})$, $S_k ( p_2, q_2, p_1, q_1)$ is given on a
neighborhood of the diagonal by the sum of the $f_{\ga , \ga }$ where
$\ga$ runs over $k^{-1} \Lat^*$. 

\begin{lem} We have
\begin{xalignat*}{2} 
 \sum_{\ga \in k^{-1}\Lat ^* }  f_{\ga,\ga}( p_2, q_2, p_1, q_1 ) = & 
 \Bigl( \frac{ik}{ \tau_2 - \overline \tau_1 } \Bigr)
^{n/2} \operatorname{Vol}( \Liet / \Lat )  \\  &   \times \sum_{\la \in \Lat} \exp
\Bigl( \frac{ - i \pi k }{ \tau_2 - \overline \tau _1} B( \la+ \zeta_2
- \overline \zeta _1, \la + \zeta_2 - \overline \zeta_1) \Bigr)
\end{xalignat*}
\end{lem}

\begin{proof} 
Introduce a basis $(\pi_i)$ of $\Lat^*$. Let $\ga \in
k^{-1} \Lat^*$, write $k\ga =  \sum x_i \pi_i  $. One has
\begin{xalignat*}{2} 
f_{\ga, \ga}( p_2, q_2, p_1, q_1 ) = & \exp \Bigl(-  \Bigl( \frac{ \tau _2 -
  \overline \tau _1}{2i} \Bigr)   \Bigl( \frac{2 \pi }{k} \Bigr) B( k\ga, k\ga) - 2 i \pi  B ( \zeta_2 - \overline \zeta _1
, k\ga ) \Bigr) \Bigr)  \\ = & u (x) \exp ( -i \langle \eta, x \rangle
)  
\end{xalignat*} 
where $u$ is the complex valued function of $\R^n$ given by 
$$ u ( x) = \exp \Bigl( - \frac{1}{2}  \Bigl(\frac{ \tau _2 -
  \overline \tau _1}{i}\Bigr) \Bigl(\frac{2 \pi}{k}\Bigr)   \sum_{i,j} B( \pi_i , \pi_j ) x_i x_j
\Bigr) $$
and 
$$ \eta_i = 2 \pi B( \zeta_2 - \overline \zeta _1, \pi_i ).$$
Let $( \ga_i)$ be the dual basis of $( \pi_i)$. Then using that $B(
\ga_i, \ga_j)$ is the inverse of $B( \pi_i, \pi_j)$ and that the
determinant of $\bigl( B( \ga_i , \ga_j ) \bigr)_{i,j}$  is the square
of the volume of $\Liet / \Lat$, we prove that the Fourier transform of
$u$ is 
$$\hat{u} ( \xi) =  \Bigl( \frac{ik}{ \tau_2 - \overline \tau_1 } \Bigr)
^{n/2} \operatorname{Vol}( \Liet / \Lat ) \exp \Bigl( -\frac{1}{2}
\Bigl( \frac{ i }{
  \tau_2 - \overline \tau _1}\Bigr) \Bigl(\frac{k}{2 \pi}\Bigr)   \sum_{i,j} B( \ga_i, \ga_j ) \xi_i
\xi_j  \Bigr).$$
The Fourier transform of $v (x) = u ( x) \exp ( - i \langle \eta , x
\rangle )$ is 
$$ \hat{ v } ( \xi ) = \hat{ u } ( \xi+ \eta  )$$
for real $\eta$. This is also verified for any $\eta \in \C^n$ by
analytic prolongation. 
By Poisson's summation formula,
\begin{gather*} 
 \sum _{x \in \Z^n} v(x)   =  \sum_{\xi \in \Z^n} \hat{v} ( 2 \pi
 \xi) 
\end{gather*}
Let us compute $\hat{v} ( 2 \pi \xi )$. Since $\sum \eta_i \ga_i = 2 \pi ( \zeta_2 - \overline \zeta
_1)$, we have 
\begin{xalignat*}{2} 
\sum_{i,j} B( \ga_i, \ga_j ) (2 \pi \xi_i + \eta_i)(2 \pi \xi_j +
\eta_j ) =   & ( 2 \pi ) ^2 B( \la + \zeta_2 - \overline \zeta
_1,  \la + \zeta_2 - \overline \zeta
_1 )  
\end{xalignat*} 
where $\la = \sum \xi_i \ga_i \in \Lat $. So $\hat{v} ( 2 \pi \xi )$
which is equal to $\hat u  ( 2 \pi \xi +
\eta)$ is given by 
$$ \hat{v} ( 2 \pi \xi ) =  \Bigl( \frac{ik}{ \tau_2 - \overline \tau_1 } \Bigr)
^{n/2} \operatorname{Vol}( \Liet / \Lat ) \exp
\Bigl( \frac{ - i \pi k }{ \tau_2 - \overline \tau _1} B( \la+ \zeta_2
- \overline \zeta _1, \la + \zeta_2 - \overline \zeta_1) \Bigr) $$
which concludes  the proof. 
\end{proof}

\begin{lem} There exists $\epsilon >0$ and $C>0$ such that 
$$  \sum_{\la \in \Lat \setminus \{0\} } \Bigl| \exp
\Bigl( \frac{ - i \pi k }{ \tau_2 - \overline \tau _1} B( \la+ \zeta_2
- \overline \zeta _1, \la + \zeta_2 - \overline \zeta_1) \Bigr)
 s_{\tau_2}^k \boxtimes
\overline{s}^k_{\tau_1}  \Bigr| \leqslant C e ^{-k/C}
$$
for all $p_2$, $q_2$, $p_1$ and $q_1$ satisfying $|q_1 - q_2 |
\leqslant \epsilon$ and $|p_1 - p_2 | \leqslant \epsilon$.
\end{lem}

\begin{proof} 
With a straightforward computation, we obtain that
$$ \Bigl| \exp
\Bigl( \frac{ - i \pi  }{ \tau_2 - \overline \tau _1} | \la+ \zeta_2
- \overline \zeta _1|^2 \Bigr) s_{\tau_2} \boxtimes
\overline{s}_{\tau_1}  \Bigr|^2 = e^{- ( a |\mu|^2 + 2 b B(\mu, q) + c |q|^2)}
$$
where $q = q_2 - q_1$ $ \mu = \la +p_2 - p_1$ and $a$, $b$ and $c$ are the real numbers
$$ a = \frac{ - i \pi }{ | \tau _2 - \overline \tau _1 |^2} ( \tau _2 -
\overline \tau _2 + \tau_1 - \overline \tau _1 ) $$
$$  b = \frac{ i \pi }{ | \tau _2 - \overline \tau _1 |^2} (
\overline \tau_2 \overline \tau_1 - \tau_2 \tau_1 )   $$
$$  c = \frac{ -i \pi }{ | \tau _2 - \overline \tau _1 |^2} ( |
\tau_2|^2 ( \tau_1 - \overline \tau_1) + | \tau_1 |^2 ( \tau_2 -
\overline \tau_2 ))$$
Write 
$$ a |\mu|^2 + 2 b B(\mu, q) + c |q|^2 = a | \mu + \tfrac{b}{a} q |^2
+ ( c - \tfrac{b^2}{a} ) |q|^2 .$$ 
Using that $a$ is positive, one proves that there exists $\epsilon >0$
such that 
$$ |q| \leqslant \epsilon \text{ and } |p_2 - p_1 | \leqslant \epsilon
\Rightarrow  \tfrac{a}{2} | \mu + \tfrac{b}{a} q |^2
+ \bigl( c - \tfrac{b^2}{a} \bigr) |q|^2 \geqslant 0  $$
for any non-vanishing $\la \in \Lat$. So 
$$  \sum_{\la \in \Lat \setminus \{0\} } \Bigl| \exp
\Bigl( \frac{ - i k  \pi  }{ \tau_2 - \overline \tau _1} | \la+ \zeta_2
- \overline \zeta _1|^2 \Bigr) s^k_{\tau_2} \boxtimes
\overline{s}^k_{\tau_1}  \Bigr| \leqslant  \sum_{\la \in \Lat \setminus
  \{0\} } e^{- \frac{a}{4} k   \bigl| \la + p_2 - p_1 + 
  \tfrac{b}{a} q \bigr|^2 } $$
when $|q|$ and $|p_2 - p_1|$ are smaller than $\epsilon$. We conclude with lemma \ref{lem:estim}. 
\end{proof}

Collecting together the previous lemmas, we obtain 
\begin{xalignat*}{2} 
S_k  & ( p_2, q_2, p_1, q_1) =   \Bigl( \frac{k}{2 \pi } \Bigr)^n \Bigl(
\frac{ 2 i \pi }{ \tau_2 - \overline \tau _1} \Bigr)^{n/2}
\operatorname{Vol} ( \Liet / \Lat )  \\  & \times \exp \Bigl( - \frac{i \pi k }{
  \tau _2 - \overline \tau _1} B( \zeta_2 - \overline \zeta_1, \zeta_2
 - \overline \zeta _1) \Bigr)  (s_{\tau_2}^k \Om^{1/2}) \boxtimes
(\overline{s}^k_{\tau_1} \overline{\Om}^{1/2}) + R_k ( p_2, q_2, p_1, q_1)
\end{xalignat*}
where the remainder satisfies for some $\epsilon >0$ and $C>0$, 
$$ |q_1 - q_2 |, | p_1 - p_2 |  \leqslant \epsilon \Rightarrow | R_ k ( p_2, q_2, p_1,
q_1) | \leqslant C e^{-k/C} .$$
Using lemma \ref{lem:E_et_f0}, this proves theorem \ref{theo:FIO}.

\section{Asymptotic properties of the quantum representations} \label{sec:asympt-prop-quant}

\subsection{Definitions}

Let $M$ be a symplectic compact manifold with a positive complex structure $j$, a prequantization
bundle $L$ and a half-form bundle $\delta$. 
Consider a symplectomorphism $\Phi : M \rightarrow M$ together with 
automorphisms $\Phi_L$ and $\varphi$ of the bundles $L$ and
$\delta$ respectively which lift $\Phi$. We assume that $\Phi_L$ preserves the
connection and metric of $L$.

% and $\varphi$ lift $\Phi$, that $\Phi_L$ is a prequantum bundle isomorphism and that the square of $\varphi$ satisfies 
% $$ \varphi_x ^{\otimes 2}  = \pi_x ^* \circ ( ( T_x \Phi)  ^*)^{-1} : \wedge ^{\operatorname{top}} E_x ^* \rightarrow \wedge ^{\operatorname{top}} E_{\Phi (x)} ^* $$ 
% Here $E_x = T_x ^{1,0} M$, $F_{\Phi (x)} = T_x \Phi ( E_x)$ so that $( ( T_x \Phi)  ^*)^{-1}$ is a linear map 
% $$ ( ( T_x \Phi)  ^*)^{-1} : \wedge ^{\operatorname{top}} E_x ^* \rightarrow  \wedge ^{\operatorname{top}} F_{\Phi(x)} ^*.$$ 
% $\pi_x$ is the linear projection from $E_{\Phi (x)}$ onto $F_{\Phi (x)}$ with kernel $\overline{E}_{\Phi (x)}$, so that 
% $$ \pi_x ^* :  \wedge ^{\operatorname{top}} F_{\Phi (x)} ^* \rightarrow  \wedge ^{\operatorname{top}} E_{\Phi(x)} ^*.$$
% We call such a $\varphi$ a half-form bundle morphism lifting $\Phi$.

Let $\Hilb_k $ be the space of holomorphic sections of
 $L^k \otimes \delta$.
Consider a family  $(S_k)$ such that for every $k$, $S_k$ is an operator $\Hilb _k
 \rightarrow \Hilb _k $. The Schwartz kernel of $S_k$ is  a holomorphic section of  
$$ (L^k \otimes \delta ) \boxtimes (\overline{L}^k \otimes \overline{\delta})
\rightarrow M^2,$$
where $M^2$ is endowed with the complex structure $(j, -j)$.
We say that $(S_k)$
is a Fourier integral operator associated to $\Phi_L$ with symbol $\varphi$ if the Schwartz kernel sequence is of the form 
\begin{gather*}
 S_k(x,y) =  \Bigl( \frac{k}{2\pi} \Bigr)^{n} F^k(x,y) g(x,y,k) + O
(k^{-\infty}) \end{gather*}
where 
\begin{itemize}
\item 
$F$ is a section of  $L \boxtimes \bar{L}
\rightarrow M^2$ such that  $\| F(x,y)
\| <1 $ if $x \neq \Phi(y) $, 
$$ F (\Phi ( x) ,x) = \Phi_L(u)  \otimes \bar{u}, \quad \forall u \in L_x \text{ such that }
  \| u \| = 1, $$
and $ \bar{\partial} F \equiv 0 $
modulo a section vanishing to any order along the graph of $\Phi ^{-1}$.
\item
  $g(.,k)$ is a sequence of sections of  $ \delta \boxtimes
  \bar{\delta}  \rightarrow M^2$ which admits an asymptotic expansion in the $\Ci$
  topology of the form 
$$ g(.,k) = g_0 + k^{-1} g_1 + k^{-2} g_2 + ...$$
whose coefficients satisfy $\bar{\partial} g_i \equiv 0  $
modulo a section vanishing to any order along the graph of $\Phi ^{-1}$.
\item  The restriction to the diagonal of the leading
coefficient $g_0$ is equal to $\varphi$, if we identify $\delta
\otimes \overline{\delta}$ with $\Hom ( \delta , \delta)$ using
the metric of $\delta$. 
\end{itemize}

Let us explain the relation with the Fourier integral operators of
section \ref{sec:def_FIO_1}. Let $\Phi(j)$ be the complex structure obtained by pushing
forward $j$ with $\Phi$. Consider a half-form bundle $\delta'$ of the
complex manifold $(M, \Phi
(j))$ together with an isomorphism $ \varphi_1 : \delta \rightarrow
\delta'$ whose square is equal to   
$$ \varphi_{1,x}^{\otimes 2} = ((T_x \Phi) ^*)^{-1} : \wedge
^{\operatorname{top},0}_j T_x ^* M \rightarrow  \wedge
^{\operatorname{top},0}_{\Phi (j)} T_{\Phi(x)} ^* M
$$
Then the isomorphisms $\Phi_L$ and $\varphi_{1}$ induce a linear isomorphism $\Phi_*$ from $\Hilb_k$ to the space
$\Hilb'_k$ consisting of the sections of $L^k \otimes \delta'$
holomorphic with respect to $\Phi(j)$. Now suppose that 
$$ S_k = T_k \circ \Phi_* : \Hilb_k \rightarrow \Hilb_k, \qquad k=1,2,
\ldots $$
for an endomorphism $T_k: \Hilb_k' \rightarrow \Hilb_k$. Then
comparing the definition of section
\ref{sec:def_FIO_1} with the previous one, we prove that $(T_k)$ is a Fourier integral
operator  with symbol $\varphi_2$ in the sense of section
\ref{sec:def_FIO_1} if and only if $(S_k)$ is a Fourier integral
operator associated to $\Phi_L$ with symbol $\varphi= \varphi_2 \circ
\varphi_1$. 
This applies to the representation $R_2$ defined in
corollary \ref{cor:def_R2}. Indeed for any $(A,e) \in
\Mo_2$, $R_2 (A,e)$ is the composition of a pull-back with the map $\Psi_{A
  \tau,\tau}$, which is a Fourier integral operator by theorem \ref{theo:FIO}. 

Since the half-form bundle $\delta_\tau$ is the trivial bundle, we can
identify its automorphisms with functions on $\mtore ^2$, the
correspondence being given by $\varphi_x ( \Om^{1/2}) = f(x)
\Om^{1/2}$. We use this convention in the sequel for the symbols of
the Fourier integral operators.

\begin{theo} \label{theo:FIO_Modul}
For any $\tau$ and $(A,e) \in \Mo_2$ the sequence $$ R_2(A,e) :
H^0_\tau ( \mtore^2 , L^k \otimes \delta ) \rightarrow H^0_\tau (
\mtore^2 , L^k \otimes \delta ) , \qquad k = 1, 2, \ldots $$
is a Fourier integral operator associated to the prequantum lift of
$A$ to $L$. Its symbol is the constant function equal to  
$$ \si (A,e) = e ( \tau) \Biggl( \frac{A \tau - \overline{A \tau }}{ \tau -
  \overline{A \tau}} \Biggr)^{n/2}. $$
\end{theo}

Applying theorem \ref{theo:FIO} with $\tau_1 = \tau_2$, the representation of the Weyl group $W$ is also given by Fourier
integral operators. 

\begin{theo} \label{theo:FIO_Weyl}
For any $\tau$ and $w \in W$ the sequence $$ w :
H^0_\tau ( \mtore^2 , L^k \otimes \delta ) \rightarrow H^0_\tau (
\mtore^2 , L^k \otimes \delta ) , \qquad k = 1, 2, \ldots $$
is a Fourier integral operator associated to the prequantum lift of
$w$ to $L$. Its symbol is the constant function equal to 1.
\end{theo}

\subsection{Metaplectic group}

Let $S$ be a symplectic vector space with a positive compatible
complex structure $j$. Denote by $E = \ker ( \id + i j )$  the space
of vectors with type $(1, 0)$. Let $\Sp (S)$ be the symplectic group
of $S$. Using the complex structure we introduce a group $\Mp (S,j)$,
isomorphic to the metaplectic group of $S$. $\Mp (S,j)$ consists of the pairs
$(A,z)$ such that $A \in \Sp (S)$ and $z$ is a complex number satisfying 
$$ z^2 = \det ( g^{-1}  \pi_{E, gE} : E \rightarrow E) .$$
Here $\pi_{E,gE}$ is the projection from $E$ onto $gE$ with kernel
$\overline E$.  The product of $\Mp (S,j)$  is
determined by the condition that the projection onto the symplectic
group is a group morphism and that the identity is the pair $(\id,
1)$. We shall also consider an extension $\Mp _2 (S, j)$ by $\Z _4 =
\{ \pm 1, \pm i \}$  of the symplectic group. It is defined as the set of pairs $(A,z)$ such that $  z^4 = \det ^ 2 ( g^{-1}  \pi_{E, gE} : E \rightarrow E)$. The product is determined by the condition that the map $$\Mp (S ,j) \times \Z_4  \rightarrow \Mp _2 (S,j)$$ sending $(A,z,u)$ into $(A, zu)$ is a group morphism.

We apply these constructions to $S = \Liet \oplus \Liet = \R^2 \otimes
\Liet $ with the complex structure given by some $\tau$ in the upper half-plane. Then the symbols of the operators defining the representation $R_2$ belongs to the metaplectic group. 
\begin{prop} \label{prop:mor-symb}
We have a group morphism from $\Mo_2$ to $\Mp ( \Liet \oplus \Liet , j_{\tau})$ sending $(A, e )$ into
$ ( A \otimes \id_\Liet ,  \sigma (A,e))$ with 
  $$ \sigma  (A,e)  = e ( \tau) \bigl( \tfrac{A \tau - \overline{A \tau }}{ \tau -
  \overline{A \tau}} \bigr)^{n/2} $$
\end{prop}
\begin{proof}
Observe that for any $g \in \Sp(S)$, the endomorphism 
$$    \pi_{E, gE}^* \circ ( g^{-1}) ^* :  \wedge ^{\operatorname{top}} E ^* \rightarrow \wedge ^{\operatorname{top}} (gE) ^* \rightarrow \wedge ^{\operatorname{top}} E ^*  $$
is the multiplication by $\det ( g^{-1}  \pi_{E, gE} : E \rightarrow
E) $. 

Let us apply this to $g =  A \otimes \id_\Liet$. By equation
(\ref{eq:1}) and the condition $e(\tau)^2 = ( - c\tau + d)^n$, the
pull-back by $g^{-1}$ is multiplication by $e^2 ( \tau)$.  By lemma
\ref{lem:half_form_morphism}, $\pi_{E, gE}^*$ is the multiplication by $\bigl( \frac{A \tau - \overline{A \tau }}{ \tau -
  \overline{A \tau}} \bigr)^{n}$. This implies that $( A \otimes \id_\Liet ,  \sigma (A,e))$ belongs
to the metaplectic group $\Mp ( \Liet \oplus \Liet , j_{\tau})$. 

 One shows that the map is a group morphism by extending it to the
 group defined as $\Mo_2$ by replacing $\Sl (2 , \Z) $ with $\Sl
 (2,\R)$ and using a continuity argument. 
\end{proof}
For any $A \in \Sl (2, \R)$, let $d(A, \tau) = \det (A ^{-1} \pi_{E, AE} : E \rightarrow E)$ with $E$ the complex polarization determined by the complex structure $p + \tau q$. Then in the proof of the previous proposition we showed that
\begin{gather} \label{eq:dAtau}
 d(A, \tau) = (-c \tau + d )  \tfrac{A \tau - \overline{A \tau }}{ \tau -
  \overline{A \tau}}. 
\end{gather}
We will use this equation several times in the sequel.
Assume that the rank of $G$ is even, so $n = 2p$. We have  a morphism from $\Mo $ into $\Mo _2$ sending $A$ into $(A, (-c\tau + d )^p)$. Composed with the morphism provided by proposition \ref{prop:mor-symb}, we obtain the group morphism
\begin{gather} \label{eq:mor_pair}
 \Mo \rightarrow \Mp ( \Liet \oplus \Liet , j_{\tau}), \qquad A \rightarrow (A  \otimes \id_\Liet , d(A, \tau) ^p ) .
\end{gather}
Assume now that the rank of $G$ is odd, $n= 2 p +1$. Introduce the subgroup $\Mp ( \Z ,  \tau)$ of $\Mp ( \R^2, \tau)$ consisting of the pairs $(A, z) \in \Mo \times \C^*$ such that $ z^2 = d (A, \tau)$. Using again (\ref{eq:dAtau}), we prove that this group is isomorphic to $\Mo_2$, the isomorphism being given by
$$ ( A, z)  \in \Mp ( \Z , \tau ) \rightarrow \Bigl( A, z ( -c \tau + d ) ^ p \bigl( \tfrac{A \tau - \overline{A \tau }}{ \tau -
  \overline{A \tau}} \bigr)^{-1/2} \Bigr) \in \Mo _2
$$
Finally composing this morphism with the one of proposition \ref{prop:mor-symb}, we obtain the group morphism
\begin{gather} \label{eq:mor_impair}
 \Mp ( \Z, \tau) \rightarrow  \Mp ( \Liet \oplus \Liet , j_{\tau}), \qquad (A,z) \rightarrow (A  \otimes \id_\Liet , z d(A, \tau) ^p ) .
\end{gather}

Considering the representation of the Weyl group, we obtain a morphism
into the extension $\Mp_2 ( \Liet \oplus \Liet , j_{\tau})$  of the symplectic group.
\begin{prop}
We have a group morphism from $W$ to $\Mp_2 ( \Liet \oplus \Liet , j_{\tau})$ sending $w$ into $( \id_{\R^2} \otimes w , 1)$ 
\end{prop}
More generally, if $w $ is an element of the orthogonal group of
$\Liet$, then $( \id_{\R^2} \otimes w, u )$ belongs to the metaplectic
group (resp. the extension by $\Z_4$) if and only if $ u^2 = \det w$
(resp. $u^4 =1$).

\subsection{Index computation}

As previously  consider the metaplectic group $\Mp (S, j)$
of a symplectic vector space endowed with a complex structure. Let
$\Mp _* (S,j)$ be the subset consisting of the pairs $(g,z)$ such that
$1$ is not an eigenvalue of $g$. Then we defined in \cite{oim_LS} an index map
$$ \ind : \Mp _* (S, j ) \rightarrow \Z / 4 \Z$$
It is continuous and takes distinct values on each of the four components of
$\Mp_* (S, j)$. To compute it, we only need the two following
properties. 
If $E$ has dimension 2, then 
\begin{gather} \label{eq:index}
 \ind( g,z) = k + \tfrac{1}{2} ( 1 - (-1) ^{k+ \epsilon} ) 
\end{gather}
where $k \in \Z $ is such that the argument of $z$ belongs to
$[\frac{\pi}{2} k , \frac{\pi}{2} ( k +1 ) [$ and $\epsilon$ is equal
to $0$ if the trace of $g$ is bigger that 2 and to $1$
otherwise. Furthermore if $(g_1, z_1) \in \Mp _*(S_1, j_1)$ and $(g_2,
z_2) \in \Mp _* (S_2, j_2)$ then $(g_1 \oplus g_2, z_1 z_2)$ belongs
to $\Mp _* ( S_1 \oplus S_2 , j_1 \oplus j_2 )$ and 
\begin{gather} \label{eq:index2}
 \ind ( g_1 \oplus g_2 , z_1 z_2 ) = \ind ( g_1, z_1) + \ind ( g_2, z_2)
\end{gather}
The elements $(g,z)$ of $\Mp_2
(S,j)$ such that $1$ is not an eigenvalue of $g$, also have an index
defined modulo $4 \Z$. It is such that
$$ \ind ( g, i^{k} z) = k + \ind (g , z)$$
if $(g,z ) \in \Mp_* (S,j)$. In the following we compute the index of
some elements of the metaplectic group of $ S = \Liet \oplus \Liet$
endowed with the complex structure determined by $\tau \in \Hp$.

\begin{lem} \label{lem:index-computation}
For any hyperbolic $A \in \Sl (2, \R)$ and $w $ in the
orthogonal group of $\Liet$, we have
$$ \ind ( A \otimes w, z ) = \ind ( A \otimes \id_{\Liet}, z)$$ 
where $z$ is any complex number such that $ (A \otimes \id_{\Liet},
z)$ belongs to $\Mp_2 ( \Liet \oplus \Liet, j_{\tau})$. 
\end{lem}

\begin{proof} 
$A$ being hyperbolic, $1$ is not an eigenvalue of $A \otimes w$. Since
$$ (A \otimes \id_{\Liet}, z ) . ( \id_{\R^2} \otimes w , u ) = ( A
\otimes w , zu)$$
The fact that $ (A \otimes \id_{\Liet},
z)$ belongs to $\Mp_2 ( \Liet \oplus \Liet, j_{\tau})$ implies that
$(A \otimes w , z )$ also belongs to $\Mp_2 ( \Liet \oplus \Liet,
j_{\tau})$. Let us prove that they have the same index.  Since the
index is locally constant, the result is straightforward if $w$
belongs to the special orthogonal group. Otherwise we may assume that
$w$ is a reflexion and that $A$
is the diagonal matrix with coefficient $2$, $1/2$. Let us decompose
$\Liet $ as a direct sum of orthogonal lines. The complex
structure $j_\tau$ preserves the associated decomposition of $\R^2
\otimes \Liet$. So using (\ref{eq:index2}) it is sufficient to  prove that 
$$ \ind ( A, u ) = \ind ( -A, u)$$ 
One may assume that $(A, u) \in \Mp ( \R^2, \tau)$ so that $(-A, iu)
\in \Mp ( \R^2, \tau)$. Then the result follows from formula (\ref{eq:index}). 
\end{proof}

We can give explicit formulas for the index of $(A \otimes \id_{\Liet},
z)$ by decomposing $\R^2 \otimes \Liet$ into a direct sum of
$\R^2$'s as we did in the previous proof. 
\begin{lem} \label{lem:ind_pair_impair}
If $n = 2p$, for any $A \in \Sl _*(2, \R)$, we have
$$ \ind ( A \otimes
\id_ {\Liet},d(A)^p) = 2 \epsilon p $$
where $\epsilon$ is equal to $0$ if the trace of $A$ is bigger than 2
and to 1 otherwise. If $n = 2p + 1$, for any $ (A, z) \in \Mp_*( \R^2,
\tau)$, we have
$$ \ind ( A \otimes
\id_ {\Liet},d(A)^pz) = 2 \epsilon  p + \ind (A,z)$$
where $\epsilon$ is defined as previously.
\end{lem}
In the second case, the index of $(A,z)$ is given by (\ref{eq:index}). With these formulas we obtain the index of any element in the images of the morphisms (\ref{eq:mor_pair}) and (\ref{eq:mor_impair}).   
\begin{proof} 
Working with the decomposition of $\R^2 \otimes \Liet$, we
only have to consider $n=2$. We have 
$$ \ind ( A \oplus A, d(A) ) = 2 \ind (A, z)$$
where $z^2 = d (A)$. We conclude with formula (\ref{eq:index}).
\end{proof}

\subsection{Trace estimates}

Under a transversality condition, the trace of a Fourier integral
operator admits an asymptotic expansion and we can explicitly compute
the leading term in terms of the symbol. Next theorem has been proved
in \cite{oim_LS}. We restrict to the case of $\mtore ^2$ to simplify the statement. 

\begin{theo} 
Let $\Phi$ be a symplectomorphism of $\mtore ^2$ whose graph intersects transversally the diagonal. Let $\Phi_L$ be a prequantum bundle isomorphism of $L$ lifting $\Phi$ and $(T_k : H^0_\tau(\mtore ^2, L^k \otimes \delta _\tau )\rightarrow H^0_\tau (\mtore ^2, L^k \otimes \delta_\tau)  ) $ be a Fourier integral operator associated to $\Phi_L$ with symbol $f$. Then for any fixed point $x$ of $\Phi$, there exists a sequence $(a_{\ell , x})$ of complex numbers such that for any $N$, 
$$ \trace (T_k) =  \sum_{x/ \; \Phi (x) =x }   u_x^k \bigl( a_{x,0} + a_{x,1} k^{-1} + \ldots + a_{x,N} k^{-N} + O(k^{-N-1}) \bigr) $$
where for any $x$, $u_x$ is the trace of $\Phi_L(x) :L_x \rightarrow L_x$. Furthermore, if $(T_x \Phi , f(x))$ is an element of $\Mp_2 ( \Liet \oplus \Liet , j_\tau)$, then
$$ a_{x,0} = \frac{i ^{\ind (T_x \Phi , f(x)) } }{|\det ( \id - T_x \Phi )|^{1/2}} .$$
\end{theo}

We apply this to estimate the character of the representation $R_2^{\alt}$. First we have 
$$ \trace ( R_2 ^{\alt} (A,e) ) = \frac{1}{|W|} \sum_{w\in W} (-1)^{\ell (w)} \trace (w. R_2 (A,e) )  $$
Then by theorems \ref{theo:FIO_Modul} and \ref{theo:FIO_Weyl}, $w.R_2
(A,e)$ is a Fourier integral operator associated to the prequantum
lift $A \otimes w$.  Its symbol is the constant
  map equal to $\si (A,e)$. By \ref{lem:index-computation}, the index
  of $(A\otimes w, \si (A,e))$ doesn't depend on $w$.  We compute
  easily the action of the prequantum lift of $A\otimes w$ at the
  fixed points and obtain the

\begin{theo} For any $(A,e) \in \Mo_2$ such that $A$ is hyperbolic, we have 
$$ \trace ( R_2^{\alt} (A,e) ) \sim \frac{i ^{n(A,e)}}{|W|}
\sum_{\substack{ w \in W \\ u \in \mtore ^2/ \; (A \otimes w).x = x }}
(-1)^{\ell (w)} 
\frac{ e ^{ ik \theta ( A \otimes w , x ) }}{ | \det ( \id - A \otimes w ) | ^{1/2}} $$
where 
\begin{itemize} 
\item $n(A,e)$ is the index of $( A \otimes \id_{\Liet}, \si (A,e))$  
\item $\theta (A \otimes w , x ) = \pi ( B( \mu, p ) - B ( \ga , q) + B( \ga , \mu))$ if $x$ is the class of  $(p,q) \in \Liet ^2$ and $(\ga, \mu) = (A \otimes w)(p,q) - (p,q)$. 
\end{itemize}
\end{theo}

We can explicitly compute the indices with lemma \ref{lem:ind_pair_impair}. For the statement in the introduction we used the two morphisms (\ref{eq:mor_pair}) and (\ref{eq:mor_impair}).

\appendix

\section{Proofs of theorem \ref{theo:base_theta} and
  \ref{theo:modular-action}} 

\subsection{The basis of $H^{0}_\tau (\mtore^2, L^k)$}

Recall that the holomorphic sections of $L^k$
identify with the sections over $\Liet^2$ of the form  $f s^k$ such that $f :
\Liet^2 \rightarrow \C$ is holomorphic and $f s^k$ is $\Lat
^2$-invariant. As shows a straightforward computation, this invariance
is equivalent to   
\begin{gather*} 
 f(p + \dot {p}, q + \dot {q} ) = f ( p, q ) \exp \bigl(- 2 i \pi k \bigl(  B( \zeta , \dot q ) +\tfrac{  \tau}{2} B( \dot q , \dot q)  \bigr) \bigr)
\end{gather*}
for all $\dot p, \dot q$ in $\Lat$. 
Then to prove that the sections $ \Theta_{\mu , k } s^k$ form a basis
of the holomorphic sections of $L^k$, we decompose the functions $f$
as a Fourier series in the $p$ variable with coefficients depending on
$q$. Then the holomorphy and the equivariance in the $q$-directions
translate into a condition on the coefficients, leading to the result.  

Let us shows that the sections $ \Theta_{\mu,k} s^k$  are mutually
  orthogonal and compute their norms. 
 Let $m$ be the Riemannian volume of $\Liet$. The
  Liouville measure $|\om ^n|/n!$ of $\Liet^2$ is equal to $(2\pi)^n m
  (p) \otimes m(q)$.  The scalar product of $f s^k $ and $f' s^k$ is given by 
\begin{gather*} 
 ( 2 \pi ) ^n  \int_{(p,q ) \in D^2 } f(p,q) \overline{f'(p,q )} \; |s(p,q)|^{2k}  m (p) \otimes m(q)
\end{gather*}
Here $D$ is the fundamental domain $\bigl\{ x_1 \mu_1 + \ldots + x_n
\mu_n /\;  (x_i) \in [0,1]^n \bigr\}$, where $(\mu_i)$
is a basis of the lattice $\Lat$.   Now a straightforward computation
shows that
$$ \Theta_{\mu, k } \overline{\Theta}_{\mu' , k } | s (p,q)|^{2k} = \sum
_{\substack{\ga \in \mu + \Lat \\  \ga' \in \mu' + \Lat }} c_{\ga, \ga'}
(q) 
\exp ( 2i \pi k B( p, \ga' - \ga ) ) $$
where the diagonal coefficients are given by $$ c_{\ga, \ga } ( q) = \exp ( i \pi k ( \tau - \bar{ \tau} ) B( q -
\ga, q - \ga )).$$ Integrating with respect to $p$, one deduces that
the scalar product of $ \Theta_{\mu, k }$ and $\Theta_{\mu' , k }$
vanishes when $\mu \neq \mu' $ mod $\Lat$. Furthermore, 
$$ \| \Theta_{\mu,k} s^k \| ^2 = ( 2 \pi ) ^n  \operatorname{Vol}(
\Liet/ \Lat) \sum_{\ga \in \mu + \Lat } \int_D c_{\ga, \ga} (q) m
(q)$$
Using that $\Liet$ is the disjoint union of the $- \ga + D$ when $\ga$
runs over $\mu + \Lat$, we obtain
\begin{xalignat*}{2}  
\sum_{\ga \in \mu + \Lat } \int_D c_{\ga, \ga} (q) m
(q) = & \int_ \Liet \exp ( i \pi k ( \tau - \bar \tau ) B( q,q ))  \; m (
q) \\
= & \Bigl( \frac{ i
} { k(\tau - \bar \tau) } \Bigr)^{n/2}
\end{xalignat*}
which ends the computation of the norm.

\subsection{The action of $S$ and $T$ in the basis of Theta functions} 

Let us prove theorem \ref{theo:modular-action}.
Denote by $\varphi_A$ the map sending $( p, q )$ into $ ( ap + bq ,
cp + dq)$. Recall that $\zeta_{\tau}  =p + \tau q$ and $s_\tau= \exp (
i \pi B( \zeta_\tau ,
q))$. Then 
$$ \varphi_A ^* \zeta_{A.\tau}  = \frac{ \zeta_\tau }{ - c \tau + d },
\quad \varphi_A ^* s_{A.\tau}  =
 \exp \Bigl( i \pi \frac{ c B( \zeta_\tau , \zeta _\tau ) } { - c\tau + d } \Bigr) s_\tau 
$$  
So 
$$ \varphi_A^* ( f (\zeta_{A.\tau}) s^k_{A.\tau} ) = f \Bigl(  \frac{ \zeta_\tau }{ - c \tau + d } \Bigr) \exp \Bigl( i \pi k
\frac{ c B( \zeta_\tau, \zeta_\tau ) } { - c\tau + d } \Bigr) s^k_\tau $$
In particular, for $A = T^{-1}$, this gives
$$\varphi_{A}^* ( f (\zeta_{A.\tau} ) s^k_{A.\tau} )
= f ( \zeta_{\tau } )  s^k_{\tau  } $$
so that $$ \varphi_{A}^* (\Theta^{A.\tau}_{\mu, k} s^k_{A.\tau} )
=  \sum_{ \ga \in \mu + \Lat} \exp \bigl( 2 i \pi k \bigl(
  \tfrac{\tau +1}{2} B( \ga, \ga ) - B( \zeta_\tau , \ga ) \bigr) \bigr) s^k_{\tau  }  $$  
Using that $B$ take integral even values on the diagonal of $\Lat^2$,
one shows that $\exp ( i k
\pi B( \ga, \ga) ) = \exp ( i k
\pi B( \mu, \mu) ) $ for any $\ga \in \mu + \Lat$. This implies
that 
$$  \varphi^*_A ( \Theta^{A.\tau}_{\mu, k}s^k_{A.\tau} ) = \sum_{ \ga \in \mu + \Lat} \exp ( i k
\pi B( \mu, \mu) ) \Theta^{\tau} _{\mu, k} s^k_{\tau  } $$
which proves the second formula of the theorem. 

Assume now that $A = S^{-1} $, then 
 $$\varphi^* ( f (\zeta_{A.\tau}) s^k_{A.\tau}  )
= f \Bigl(\frac{\zeta_\tau}{\tau} \Bigr) \exp \Bigl( - \frac{i
  \pi k } { \tau} B( \zeta_\tau, \zeta_\tau)  \Bigr) s^k_\tau$$
Applying to the theta functions, we have
$$ \varphi^* ( \Theta^{A.\tau}_{\mu, k } s^k_{A.\tau}) = \sum_{\ga \in \mu + \Lat} \exp
\Bigl( - \frac{i \pi k }{ \tau} B( \ga + \zeta_{\tau}, \ga + \zeta_{\tau} ) \Bigr)
\; s^k_{\tau} $$
Applying Poisson summation formula, we obtain after some computations that $\varphi^* (
\Theta_{\mu, k }^{A.\tau} s^k_{A.\tau})$ is equal to 
$$  \Bigl( \frac{ \tau}{ik}
\Bigr) ^{ n/2} \operatorname{Vol} ( \Liet / \Lat )^{-1}
\sum_{\ga  \in k^{-1} \Lat^* } \exp \bigl( 2 i \pi k ( \tfrac{ \tau} {2} B (
\ga, \ga ) - B( \mu + \zeta_{\tau}, \ga ) ) \bigr) s ^k_{\tau} 
$$
Since $k^{-1} \Lat^* = \bigcup (\mu' + \Lat)$, where $\mu' $ runs over
$k^{-1} \Lat^*$ mod $\Lat$, this is equal to
$$ \Bigl( \frac{ \tau}{ik}
\Bigr) ^{ n/2} \operatorname{Vol} ( \Liet / \Lat )^{-1}
\sum_{\substack {\mu'
    \in k^{-1} \Lat^* \\ \operatorname{mod} \Lat}} \exp(-  2 i \pi k B(
\mu, \mu') )  \Theta^{  \tau} _{\mu ', k
} \; s^k_{\tau}$$
which ends the proof.

\section{Index}

Lie group notations: 
\begin{tabbing}
\indent \= $M_\mu$, $M_0=M G$\quad \= \kill % sample line
\> $\Lie$, $B$ \> Lie algebra of $G$ and its basic inner product;
\ref{sec:notation} \\
\> $\mtore$, $\Liet$, $\Lat$ \> maximal torus, its Lie algebra and integral Lattice; \ref{sec:notation}\\
\> $\Alc$, $W$  \> open fundamental Weyl alcove and Weyl group; \ref{sec:notation}\\
\>  $\ell:W \rightarrow \{ \pm 1 \}$  \> alternating character; \ref{sec:notation}
\end{tabbing}

\noindent Moduli space and its quantization:
\begin{tabbing}
\indent \= $M_\mu$, $M_0=M G$\quad \= \kill % sample line
\> $p, q$, $\zeta = p + \tau q $ \> projections from $\Liet^2$ onto
$\Liet$, complex coordinates; \ref{sec:preq-bundle-Liet2}, \ref{sec:compl-struct-theta} \\ 
\> $\om$, $L_{\Liet^2}$ \> symplectic form and prequantum bundle of
$\Liet^2$;   \ref{sec:preq-bundle-Liet2} \\
\> $L$  \> prequantum bundle of $\mtore^2$; \ref{sec:heis-group-reduct} \\
\> $s$, $\Theta_{\mu,k}$ \> section of $L_{\Liet^2}$ and theta
function; \ref{sec:compl-struct-theta} \\ 
\> $ (\chi_{\mu,k})_\mu$ \> basis of alternating sections of $H^0_\tau (
\mtore^2, L^k)$; \ref{sec:alternating-sections}  
\end{tabbing}

\noindent Modular group extensions and their representation
\begin{tabbing}
\indent \= $M_\mu$, $M_0$\quad \= \kill % sample line
\> $\Mo_2$, $R_2$ \> extension by $\Z_2$ of $\Mo$,
representation in $H^0_\tau(\mtore^2, L^k)$; \ref{sec:modular-action} \\  
\> $ R_2^{\alt}$  \>  representation in the subspace of alternating
sections; \ref{sec:modular-action} \\
\> $ \Mo_\infty$, $R_\infty$ \> extension by $\Z$ of $\Mo$ and its
representation; \ref{sec:rep_infty}
\end{tabbing}

\bibliography{biblio}

\def\cprime{$'$}
\begin{thebibliography}{10}

\bibitem{BaKi}
Bojko Bakalov and Alexander Kirillov, Jr.
\newblock {\em Lectures on tensor categories and modular functors}, volume~21
  of {\em University Lecture Series}.
\newblock American Mathematical Society, Providence, RI, 2001.

\bibitem{oim_qm}
L.~Charles.
\newblock Quasimodes and {B}ohr-{S}ommerfeld conditions for the {T}oeplitz
  operators.
\newblock {\em Comm. Partial Differential Equations}, 28(9-10):1527--1566,
  2003.

\bibitem{oim_MCG}
L.~Charles.
\newblock Asymptotic properties of the quantum representations of the mapping
  class group, 2010.

\bibitem{oim_LS}
L.~Charles.
\newblock A {L}efschetz fixed point formula for symplectomorphisms, 2010.

\bibitem{Fr}
Daniel~S. Freed.
\newblock Remarks on {C}hern-{S}imons theory.
\newblock {\em Bull. Amer. Math. Soc. (N.S.)}, 46(2):221--254, 2009.

\bibitem{Je}
Lisa~C. Jeffrey.
\newblock Chern-{S}imons-{W}itten invariants of lens spaces and torus bundles,
  and the semiclassical approximation.
\newblock {\em Comm. Math. Phys.}, 147(3):563--604, 1992.

\bibitem{Ki}
Alexander~A. Kirillov, Jr.
\newblock On an inner product in modular tensor categories.
\newblock {\em J. Amer. Math. Soc.}, 9(4):1135--1169, 1996.

\bibitem{ReTu}
N.~Reshetikhin and V.~G. Turaev.
\newblock Invariants of {$3$}-manifolds via link polynomials and quantum
  groups.
\newblock {\em Invent. Math.}, 103(3):547--597, 1991.

\bibitem{Tu}
V.~G. Turaev.
\newblock {\em Quantum invariants of knots and 3-manifolds}, volume~18 of {\em
  de Gruyter Studies in Mathematics}.
\newblock Walter de Gruyter \& Co., Berlin, 1994.

\bibitem{Wi}
Edward Witten.
\newblock Quantum field theory and the {J}ones polynomial.
\newblock {\em Comm. Math. Phys.}, 121(3):351--399, 1989.

\end{thebibliography}
 
\end{document}